\documentclass[twoside]{irmaems}
\usepackage{amsmath,amssymb,enumerate,color}
 \usepackage[usenames,dvipsnames]{pstricks}
 \usepackage{epsfig}
 \usepackage{hyperref}

 \usepackage[applemac]{inputenc} % (pour les accents)
 
\numberwithin{equation}{section}   
\newcommand{\ds}{\displaystyle}
\newcommand{\tq}{\, \big| \, }
\renewcommand{\r}{\mathbb{R}}
\newcommand{\dfunk}{\varrho_f}
\newcommand{\dhilb}{\varrho_h}
\newcommand{\nosymbol}{ \  }

\DeclareMathOperator{\g}{g} 
\DeclareMathOperator{\RR}{R} 
\DeclareMathOperator{\Riem}{Riem} 
\DeclareMathOperator{\PP}{Proj} 
\DeclareMathOperator{\K}{K} 
\DeclareMathOperator{\Sc}{\mathcal{R}} 
\DeclareMathOperator{\Tau}{\mathcal{T}}

\makeatletter
\renewenvironment{proof}[1][\proofname]{\par
  \pushQED{\qed}%
  \normalfont \topsep6\p@\@plus6\p@\relax
  \trivlist
  \itemindent0pt
  \item[\hskip\labelsep
        \scshape
    #1\@addpunct{.}]\ignorespaces
}{%
  \popQED\endtrivlist\@endpefalse
}
\makeatother

\definecolor{NoteColor}{rgb}{1,0,0}

\newtheorem*{theoremA}{\rm\bf Theorem} 
\newtheorem{theorem}{\rm\bf Theorem}[section]
\newtheorem{proposition}[theorem]{\rm\bf Proposition}
\newtheorem{lemma}[theorem]{\rm\bf Lemma}
\newtheorem{corollary}[theorem]{\rm\bf Corollary}

\theoremstyle{definition}
\newtheorem{definition}[theorem]{\rm\bf Definition}

\theoremstyle{remark}
\newtheorem{remark}[theorem]{\rm\bf Remark}

\newtheorem{example}[theorem]{\rm\bf Example} 
\newtheorem{examples}[theorem]{\rm\bf Examples}

\numberwithin{equation}{section}

\title{Funk and Hilbert geometries from the Finslerian Viewpoint}

\author{Marc Troyanov}
\address{M. Troyanov, Section de Math{\'e}matiques,  \'Ecole Polytechnique F{\'e}d\'erale de
Lausanne, \\ SMA--Station 8, 1015 Lausanne - Switzerland \ (marc.troyanov@epfl.ch)}
 
\date{\today}
\makeindex 
 
\begin{document}

\maketitle

\begin{abstract}
In 1929, Paul Funk and Ludwig Berwald gave a characterization of Hilbert geometries from the Finslerian 
viewpoint. They  showed that a smooth Finsler metric in a  convex bounded  domain 
of $\r^n$ is the Hilbert geometry in that domain if and only if it is complete, if its geodesics are straight
lines and if its flag curvature is equal to $-1$. The goal of this chapter is to explain these notions  in details,
to illustrate the relation between Hilbert geometry, Finsler geometry and the calculus of variations, 
and to prove the Funk-Berwald characterization Theorem.
\end{abstract}

\tableofcontents

%__________________
\section{Introduction}

The hyperbolic space $\mathbb{H}^n$ is a complete, simply connected Riemannian manifold with constant sectional curvature $\K = -1$. It is unique up to isometry and several concrete models are available and well known. 
In particular, the \emph{Beltrami-Klein} model is a realization in the unit ball $\mathbb{B}^n \subset \r^n$ in which  the
hyperbolic lines are represented by the affine segments joining pairs of points on the boundary sphere $\partial\mathbb{B}^n$ and the distance between two points $p$ and $q$ in $\mathbb{B}^n$ is one half the logarithm of the cross ratio of $a,p$ with $b,q$, where $a$ and $b$ are the intersections of the line through $p$ and $q$ with  $\partial\mathbb{B}^n$.  
We refer to \cite{AcampoPapadopoulos} for  a historical discussion of this model.
In 1895, David Hilbert observed that the Beltrami-Klein construction defines a distance in any convex set $\mathcal{U} \subset \r^n$, and this metric still has the properties that affine segments are the shortest curves connecting two points. More precisely, if the point $x$ belongs to the segment $[p,q]$, then
\begin{equation}\label{eq.proj0}
  d(p,q) = d(p,x) + d(x,q).
\end{equation}
Metrics satisfying this property are said to be \emph{projective} and Hilbert's $\mathrm{IV^{th}}$ problem 
asks to classify and describe all projective metrics in a given domain $\mathcal{U} \subset \r^n$, see \cite{Papadopoulos-Hilbert}.
 
\smallskip
 
The Hilbert generalization of Klein's model to an arbitrary convex domain $\mathcal{U} \subset \r^n$ is no longer a 
Riemannian metric, but it is a  \emph{Finsler metric}. 
This means that the distance between two points $p$ and $q$ in $\mathcal{U}$ is the infimum of the length of all smooth curves
$\beta : [0,1] \to \mathcal{U}$ joining these two points, where the length is given by an integral of the type 
\begin{equation}\label{eq.Calcvar}
 \ell (\beta) = \int_0^1 F(\beta (t), \dot \beta (t)) dt.
\end{equation}
Here  $F : T\mathcal{U} \to \r$  is a sufficiently regular function called the \emph{Lagrangian}. Geometrical considerations lead us to 
assume that $F$ defines a norm on the tangent space at any point of $\mathcal{U}$;  in fact it is useful to  consider also non-symmetric 
and possibly degenerate norms (which we call \emph{weak Minkowski norms}, see \cite{PT2013a}).
 
  \medskip
 
General integrals of the type (\ref{eq.Calcvar}) are the very subject of the classical calculus 
of variations and Finsler geometry is really a daughter of that field.  The early
contributions (in the period 1900-1920) in Finsler geometry are due to mathematicians working in the calculus of variations\footnote{Although Bernhard Riemann already considered
this possible generalization of his geometry in his Habilitation Dissertation, he did not pursue the subject and Finsler geometry did not
emerge as a subject before the early twentieth century.}, in particular Bliss, Underhill, Landsberg, Hamel, Carath\'eodory and his student Finsler. Funk is the author of a quite famous book on the calculus of variations that contains a rich chapter on Finsler geometry \cite{Funk1970}. The name ``Finsler geometry''  has been proposed (somewhat improperly) by 
\'Elie Cartan in 1933. The 1950 article \cite{Busemann1950} by H. Busemann contains some interesting historical remarks
on the early development of the subject, and the 1959 book by H. Rund \cite{Rund59} gives a broad overview of the development of Finsler geometry  during its first 50 years; that book is  rich in references and historical comments.

\medskip

In 1908, A. Underhill  and G. Landsberg introduced a notion of curvature for two-dimensional Finsler manifolds that generalizes the classical Gauss curvature of Riemannian surfaces,
and in 1926  L. Berwald generalized this construction to higher dimensional Finsler spaces \cite{Berwald1926,Landsberg1908,Underhill}.
This invariant is nowadays called the \emph{flag curvature} and it generalizes the Riemannian 
sectional curvature. In 1929, Paul Funk proved that the flag curvature of a Hilbert geometry in a (smooth and strongly convex) domain $\mathcal{U} \subset \r^2$ is  constant $\K =-1$ and Berwald extended this result in all dimensions and refined Funk's investigation in several aspects, leading to the following  characterization of Hilbert geometry \cite{Berwald1929,Funk1929}:

\begin{theoremA}[Funk-Berwald]  \index{theorem!Funk-Berald}
Let $F$ be a smooth and strongly convex Finsler metric defined in a bounded convex domain  
$\mathcal{U} \subset \r^n$. Suppose that 
$F$ is  complete with  constant flag curvature $\K =-1$ and that 
the associated distance $d$ satisfies (\ref{eq.proj0}), then $d$ is the Hilbert metric in $\mathcal{U}$.
\end{theoremA}

The completeness hypothesis means that every geodesic segment can be indefinitely extended in both directions, in other words
every geodesic segment is contained in a line that is isometric to the full real real line $\r$. 

In fact, Funk and Berwald only assumed the Finsler metric to be reversible, that is $F(p,-\xi) = F(p,\xi)$. But they also needed the (unstated)
hypothesis that $F$ is  \emph{forward complete}, meaning that  every oriented geodesic segment can be  extended as a ray. 
Observe that reversibility together with forward completeness implies the completeness of $F$, and therefore the way we characterize the Hilbert metrics in the above theorem is
slightly stronger than the Funk-Berwald original statement. 
Note that completeness is necessary:  it is possible to locally construct reversible (incomplete) Finsler metrics for which the condition  
(\ref{eq.proj0}) holds and which are not restrictions of Hilbert metrics.

\medskip

The Berwald-Funk Theorem is  quite remarkable, and it  has been an  important landmark in Finsler geometry.
Our goal in this chapter is to develop all the necessary concepts and tools to explain this statement precisely. The 
actual proof is given in  Section \ref{sec.charH}.

 \medskip
 
The rest of the chapter is organized as follows. In Section 2 we explain what a Finsler manifold is and  give some basic definitions in the subject with a few elementary examples. In Section 3 we introduce a very natural
example of Finsler structure in a convex domain, discovered by Funk, and that we called the \emph{tautological Finsler structure}. We also compute the distances and the geodesics for the tautological structure. In Section 4 we introduce  Hilbert's Finsler structure as the symmetrization of the tautological structure and compute its distances and geodesics.

In Section 5, we introduce the \emph{fundamental tensor} $g_{ij}$ of a Finsler metric (assuming some smoothness and strong convexity hypothesis). 
In Section 6 we compute the geodesic equations and we introduce the notion of \emph{spray} and the exponential map. In Section 7 we discuss various characterizations of \emph{projectively flat} Finsler metrics, that is Finsler metrics for which the affine lines are geodesics. In Section 8 we introduce the Hilbert differential form and the notion of Hamel potential, which is a tool to compute distances in a general projectively flat Finsler space.

In Section 9 we introduce the curvature of a Finsler manifold based on the classical  Riemannian curvature of some associated (osculating) Riemannian metric.
The curvature of projectively flat Finsler manifolds is computed in Sections 10 and 11. The characterization of
Hilbert geometries is given in Section 12 and the chapter ends with two appendices: one on further developments of the subject and one on the Schwarzian derivative. 

\medskip

We tried to make this chapter as  self-contained as possible.
The material we present here is essentially built on Berwald's  paper  \cite{Berwald1929}, but we have also 
used a number of more recent sources. 
In particular we found the books \cite{Shen2001a,Shen2001b}
by Z. Shen and   \cite{ChernShen}  by S.-S. Chern and Z. Chen quite useful.

%_______________________________ 
\section{Finsler manifolds}  

We start with  the main definition of this chapter:

\begin{definition}  \index{Finsler structure}
 \textbf{(A)} A \emph{weak Finsler structure}\index{Finsler structure}\index{weak Finsler structure} on a smooth manifold $M$ is given by a  lower semi-continuous 
 function $F:TM\to [0,\infty]$ such that for every point $x\in M$ the  restriction $F_x = F_{|T_xM}$ is a  weak Minkowski norm,
 that is,  it satisfies the following properties:
 \begin{enumerate}[i.)]
 	\item $F(x, \xi_1+\xi_2) \leq F(\xi_1)+ F(\xi_2)$,
	\item $F(x,\lambda \xi)=\lambda F(\xi)$  for all $\lambda\geq 0$,
\end{enumerate}
 for any $x\in M$ and $\xi, \eta \in T_xM$.
 
 \smallskip
 
\textbf{(B) } If $F:TM\to [0,\infty)$ is finite and continuous  and if $F(x,\xi) >0$ for any $\xi \neq 0$ in the tangent space $T_xM$,
then one says that $F$ is a \emph{Finsler structure} on $M$.
\end{definition}

\medskip

We shall mostly be interested in Finsler structures, but extending the theory to the case of weak Finsler structures can be 
useful in some arguments. Another situation where weak Finsler structures appear naturally is the field of \emph{sub-Finsler}
geometry.

\medskip
 
The function $F$ itself is called the \emph{Lagrangian} of the Finsler structure, 
it is also called the (weak) \emph{metric}, since it is used to measure the length of a tangent vector. 

The weak Finsler metric is called  \emph{reversible} if $F(x, \cdot )$  is actually a norm, that is if it satisfies
 $$F(x,-\xi) = F(x,\xi)$$
  for all point $x$ and any tangent vector $\xi  \in T_xM$. 
 The Finsler metric is said to be \emph{of class $C^k$} if  the restriction of $F$ to the slit tangent bundle
$TM^0 = \{(x,\xi) \in TM  \tq \xi \neq 0\}\subset TM$ is a function of class  $C^k$. If it is $C^{\infty}$ on $TM^0$,
then one says $F$ is smooth.
 
\medskip

Some of the most elementary examples of Finsler structures are

\begin{examples}
\begin{enumerate}[i.)]
  \item A weak Minkowski norm $F_0 : \r^n \to \r$ defines a weak  Finsler structure $F$ on $M = \r^n$ by 
  $$F(x,\xi) = F_0(\xi).$$
  See \cite{PT2013a} for a discussion of weak Minkowski norms.
  This Finsler structure is constant in $x$ and is thus  translation invariant. Such structures are considered to be the
  flat spaces of Finsler geometry.
  \item A Riemannian metric $\g$ on the manifold $M$ defines a Finsler structure on $M$ by $F(x,\xi)= \sqrt{\g_x(\xi,\xi)}$.
   \item If $\g$ is  Riemannian metric on $M$ and $\theta \in \Omega^1(M)$ is a $1$-form whose 
   $g$-norm is everywhere smaller than $1$, then one can define a smooth Finsler structure on $M$ by
  $$
    F(x,\xi) = \sqrt{\g_x(\xi,\xi)}  +  \theta_x(\xi).
  $$
  Such structures are called 
  \emph{Randers metrics}.\index{Randers metrics}
  \item If $\theta$ is an arbitrary $1$-form on the Riemannian manifold $(M,g)$, then one can define
  two new Finsler structures on $M$ by
  $$
    F_1(x,\xi) = \sqrt{\g_x(\xi,\xi)}  +  |\theta_x(\xi)|, \qquad 
     F_2(x,\xi) = \sqrt{\g_x(\xi,\xi)}  + \max \{ 0, \theta_x(\xi)\}.
  $$
 \item If $F$ is  Finsler structure on $M$, then the \emph{reverse Finsler structure} $F^*:TM\to [0,\infty)$ 
 is defined by
 $$
   F^*(x,\xi) = F(x, -\xi).
 $$
\end{enumerate}
\end{examples}

Additional examples will be given below.

\medskip

A number of concepts from Riemannian geometry naturally extend to Finsler geometry, in particular one defines the \emph{length} of a  smooth curve $ \beta  : [0,1] \to M$  as 
$$\ell (\beta ) = \int_0^1 F(\beta (s) , \dot \beta (s)) dt.$$
 We  then define the \emph{distance} $d_F(x,y)$ between two points $p$ and $q$ to be  the infimum of the
length of all smooth curves  $\beta:  [0,1] \to M $ joining these two points (that is $\beta(0)=p$ and  $\beta(1)=q$). 
This distance satisfies the axioms of a weak metric, see \cite{PT2013a}. 
Together with the distance comes the notion of completeness. 
 
\begin{definition}\label{def.completeness}
 The Finsler Manifold $(M,F)$ is said to be \emph{forward complete} if every forward Cauchy sequence converges. A sequence $\{ x_i\}\subseteq M$ is \emph{forward Cauchy} if for any 
 $ \varepsilon >0$, there exists an integer $N$ such that $d(x_i,x_{i+k}) <  \varepsilon$ for any $i \geq N$ and $k \geq 0$. 
We similarly define \emph{backward Cauchy sequences} by the condition $d(x_{i+k},x_i) <  \varepsilon$, and the corresponding  notion of \emph{backward completeness}. The Finsler Manifold $(M,F)$ is  \emph{complete} if it is
both backward and forward complete.
\index{forward complete Finsler manifold}
\index{backward complete Finsler manifold}
\index{complete Finsler manifold}
\end{definition}

\medskip

A weak Finsler structure on the manifold $M$ can also be seen as a ``field of convex sets'' sitting in the tangent bundle $TM$ of that manifold.
Indeed, given a Finsler structure $F$ on the manifold $M$, we define its \emph{domain} $\mathcal{D}_F \subset TM$ to be
the set of all vectors with finite $F$-norm. The  \emph{unit domain} $\Omega \subset \mathcal{D}_F$
is the bundle of all tangent unit balls
$$
 \Omega = \{(x,\xi) \in TM \tq   F(x,\xi) < 1 \}.
$$
The unit domain  contains the zero section in $TM$ and its restriction 
$\Omega_x = \Omega \cap T_xM$ to each tangent space is a bounded and convex set. It is called the \emph{tangent unit ball}
of $F$ at $x$, while its boundary 
$$\mathcal{I}_x = \{\xi \in T_xM \tq   F(x,\xi) = 1 \}  \subset T_xM$$
 is called the \emph{indicatrix}.\index{indicatrix} of $F$ at $x$.
We know from \cite[Theorem 3.12]{PT2013a} that the Lagrangian  $F : TM \to \r$  can be recovered from $\Omega \subset TM$ via the formula
\begin{equation}\label{eq.DefLag}
  F(x,\xi) =  \inf \{ t > 0  \tq  \tfrac{1}{t}\xi \in \Omega \}.
\end{equation}

\smallskip

Sometimes, a weak Finsler structure is defined by specifying its unit domain, and the Lagrangian is then 
obtain from (\ref{eq.DefLag}). Let us  give two elementary examples; more can be found in \cite{PT1}:

\begin{example} Given  a bounded open set $\Omega_0 \subset \r^n$ that contains the origin, 
one naturally defines a Finsler structure on $\r^n$
by parallel transporting  $\Omega_0$. That is, $\Omega \subset T\r^n = \r^n \times \r^n$ is defined as
$$
 \Omega = \{ (x, \xi) \in  \r^n \times \r^n \tq  \xi \in \Omega_0 \}.
$$
The space $\r^n$ equipped with this Finsler structure is characterized by the property  that its Lagrangian $F$  is 
invariant with respect to the  translations of $\mathbb{R}^n$, i.e. $F$   is independent of the point $x$:
$$
  F(x,\xi) = F_0(\xi).
$$
Such a Finsler  structure on $\r^n$ is of course  the same thing as  a Minkowski norm, see  \cite{PT2013a} 
\end{example}
 
\begin{example} \label{ex:zerm}  
Let $F$ be an arbitrary weak Finsler structure on a manifold $M$ with unit domain $\Omega \subset TM$.
If  $Z : M \to TM$ is a continuous  vector field such that
$F(x,Z(x)) < 1$ for any point $x$ (equivalently $Z(M) \subset \Omega$), then a new weak Finsler structure can be defined as 
\begin{equation}\label{eq.zermelo0}
 \Omega_Z = \{(x,\xi) \in TM \tq \xi \in (\Omega _x - Z(x)) \}.
\end{equation}
Here $\Omega _x - Z(x)$ is the translate of $\Omega _x \subset T_xM$ by the vector $-Z(x)$.
The corresponding Lagrangian is given by
\begin{equation}\label{eq.zermelo1}
 F_Z(x, \xi)  = \inf \{ t > 0  \tq \tfrac{1}{t} \xi \in  (\Omega_x - Z(x))\},
\end{equation}
and for $\xi \neq 0$, it is computable from the identity
\begin{equation}\label{eq.zermelo2}
  F\left(x, \frac{\xi}{F_Z(x, \xi)} + Z(x) \right) = 1.
\end{equation}
\end{example} 
This weak Finsler structure $F_Z$  is called the \emph{Zermelo transform} of 
$F$ with respect to the vector Field $Z$.

\medskip

We end this section with a word on Berwald spaces. Recall first that in  Riemannian geometry, the 
tangent space of the manifold at every point is isometric to a fixed
model, which is a Euclidean space. Furthermore, using the Levi-Civita connexion
one  defines  parallel transport along any piecewise $C^1$ curve, and this parallel transport induces an isometry between the tangent spaces
at any point along the curve. In a general Finsler manifold, neither of these facts holds. This motivates the following definition:

\begin{definition}\label{def.berwald} 
A weak Finsler  manifold $(M,F)$ is  said to be  \emph{Berwald}\index{Berwald metric}  if there exists a  torsion free linear  connection $\nabla$ (called an  \emph{associated connection}) on $M$ whose   associated parallel transport  preserves the Lagrangian $F$. That is, if $\gamma:[0,1]\to M$ is a smooth path connecting the point $x= \gamma(0)$ to $y= \gamma(1)$  and $P_{\gamma} : T_xM \to T_yM$  is the associated $\nabla$-parallel transport, then   
$$
  F(y, P_{\gamma}(\xi)) = F(x, \xi)
$$
for any $\xi \in T_xM$. 
\end{definition}

\medskip

This definition is a slight generalization of the usual definition, compare with \cite[Proposition 4.3.3]{ChernShen}.  
It is known that the connection associated  to a smooth Berwald metric can be chosen to be the Levi-Civita of some Riemannian metric \cite{MatveevTroyanov,Sz1}.

\section{The tautological  Finsler structure}

\begin{definition} Let us consider a proper convex set  $\mathcal{U} \subset \r^n$. This will be our ground manifold. 
The  \emph{tautological weak Finsler  structure} $F_f$ on $\mathcal{U}$ is the Finsler structure 
for which the unit ball at a point $x\in \mathcal{U}$ is the domain  $\mathcal{U}$  itself, but with the point $x$ as center.
The unit domain of the tautological weak Finsler  structure  is thus defined as
$$
  \Omega = \{(x,\xi) \in T\mathcal{U}   \tq \xi \in (\mathcal{U}  - x)\}  \subset T\mathcal{U} = \mathcal{U} \times \r^n,
$$ 
and the Lagrangian is given by 
$$
 F_f(x,\xi) = \inf \{t>0 \tq \xi \in t (\mathcal{U}-x)\} =  \inf \left\{t>0 \tq \left(x+\frac{\xi}{t}\right)\in \mathcal{U} \right\}.
$$
Equivalently, $F_f$ is given by $F_f(x,\xi) = 0$ if the ray $x+\r_+\xi$ is contained in $\mathcal{U}$ and 
$$
 F_f(x,\xi) > 0 \quad \text{and} \quad \left(x+\frac{\xi}{F_f(x,\xi) } \right)  \in \partial \mathcal{U}
$$
otherwise.
\end{definition}
The convex set $\mathcal{U}$ can   be recovered from the Lagrangian as follows:
\begin{equation}\label{eq.RecovU}
 \mathcal{U} = \{ z \in \r^n  \tq  F_f(x,z-x) < 1\}.
\end{equation}
By construction this formula is independent of $x$.
The tautological weak structure has been introduced by  Funk in 1929, see \cite{Funk1929}.
We will often call it the \emph{Funk weak metric}
on $\mathcal{U}$  (whence the index $f$ in the notation $F_f$).

\begin{remark}
If the convex set $\mathcal{U}$ contains the origin, then the  tautological weak structure 
 $F_f$ is the Zermelo transform of the Minkowski norm with unit ball $\mathcal{U}$ for 
the position vector field $Z_x = x$
\end{remark}

\begin{example}\label{example.ball} 
Suppose $\mathcal{U}$ is the Euclidean unit ball $\{x \in  \r^n \tq \|x\| < 1\}$, where $\| \cdot \|$ is the Euclidean norm. 
Then
 $$
  (x+\tfrac{\xi}{F_f})  \in \partial \mathcal{U}  \quad \Longleftrightarrow \quad  \left\|x + \tfrac{\xi}{F_f}\right\|^2= 1.
 $$
 Rewriting this condition as
 $$
  F_f^2 (1-\|x\|^2) - 2F\cdot \langle x,\xi \rangle - \|\xi\|^2=0
 $$
the non-negative root of this quadratic equations is given by
\begin{equation}\label{taut.ball}
  F_f(x,\xi) =\frac{\langle x,\xi \rangle  + \sqrt{\langle x,\xi \rangle^2 +(1-\|x\|^2)\|\xi\|^2}}{(1-\|x\|^2) },
\end{equation}
which is the Lagrangian of the tautological Finsler structure in the Euclidean unit ball. 
Observe that this is a Randers metric.
\end{example}    
 
\medskip

\begin{example}\label{example.halfspace}
 Consider a half-space $\mathcal{H} = \{x \in  \r^n \tq \langle \nu, x\rangle < \tau\} \subset \r^n$ where $\nu$
 is a non-zero vector and $\tau\in \r$;  here $ \langle \nosymbol, \nosymbol \rangle$ is the standard scalar product in $\r^n$.
We then have
 $$
   x+\frac{\xi}{F } \in \partial \mathcal{H}  \quad \Longleftrightarrow \quad   \langle \nu, x + \frac{\xi}{F}Ê\rangle = \tau,
 $$
which implies that
\begin{equation}\label{eq.FunkHalfspace}
   F_f(x,\xi) = \max \left(\frac{\left\langle \nu, \xi\right\rangle }{\tau-\langle \nu, x\rangle } , 0\right).
\end{equation}
\end{example} 

\medskip

We now compute the distance between two points in the tautological Finsler structure:

\begin{theorem}\label{th.tautdist}
 The tautological distance in a proper convex domain $\mathcal{U}\subset \r^n$ is given  by 
 \begin{equation}\label{taut.distd}
 \dfunk(p,q) = \log \left(\frac{|a-p|}{|a-q|}\right).
\end{equation}
 where $a$ is the intersection of the ray starting at $p$ in the direction of $q$ with the boundary $\partial \mathcal{U}$:
$$
 a =  \partial \mathcal{U} \cap \left( p+ \r_+ (q-p) \right).
$$
If  the  ray is contained in $\mathcal{U}$, then $a$ is considered to be a point at infinity and $\dfunk(p,q) = 0$.
\end{theorem}

\medskip

\begin{definition}
 The distance (\ref{taut.distd}) is called the \emph{Funk metric}\index{Funk metric}   in $\mathcal{U}$.  
 In his paper \cite{Funk1929}, Funk introduced the 
 tautological Finsler structure while the distance (\ref{taut.distd}) appears as an interesting geometric object  in the 
 1959 memoir \cite{Zaustinsky}  by Eugene Zaustinsky, a  student of Busemann. We refer to  \cite{WeakFinsler, PT2013b,Zaustinsky} for some discussions
 on Funk geometry.
 \end{definition}

\begin{proof}  We follow the proof in \cite{WeakFinsler}. Consider first the  special case where $\mathcal{U} = \mathcal{H}$ is 
the half-space $ \{x \in \r^n \tq \langle \nu, x\rangle < \tau\}\subset \r^n$ for some vector $\nu \neq 0$. 
Then the tautological Lagrangian is given by (\ref{eq.FunkHalfspace}) and  the length of a curve $\beta$ joining $p$  to $q$ is
\begin{eqnarray*}
 \ell (\beta) &=&\int_0^1\max\left\{0,  \frac{\langle \nu, \dot \beta (s)\rangle }{s-\langle \nu, \beta (s) \rangle }\right\} {dt} 
=
\int_0^1\max\left\{0,  \frac{(\tau-\langle \nu,  \beta (s)\rangle)'   }{|\tau-\langle \nu, \beta (s) \rangle |}\right\} {dt}
\\ &\geq&  \max\left\{0,  \log\left(\frac{\tau-\langle \nu, p \rangle}{\tau-\langle \nu,q\rangle}\right)\right\},
\end{eqnarray*}
with equality if and only if $\langle \nu, \dot \beta (s)\rangle $ has almost everywhere constant sign. 
It follows that the tautological distance in the half-space $\mathcal{H}$ is given by
\begin{equation}\label{taut.dista}
 \dfunk(p,q) =  \max\left\{0,  \log\left(\frac{\tau-\langle \nu, p \rangle}{\tau-\langle \nu,q\rangle}\right)\right\}.
\end{equation}

\medskip

We can rewrite this formula in a different way. Suppose that the ray $L^+$ starting at $p$ in the direction of $q$
meets the hyperplane $\partial \mathcal{H} $ at a point $a$. Then $\langle \nu, p \rangle = \tau$ and  
$$
 \tau- \langle \nu, p \rangle = \langle \nu, a-p \rangle = |a-p| \cdot \langle \nu, \xi \rangle,
$$
where $\xi = \frac{a-p}{|a-p|}$. We also have 
$\tau- \langle \nu,q \rangle =  |a-q| \cdot \langle \nu, \xi \rangle$;  therefore
\begin{equation}\label{taut.distb}
 \dfunk(p,q) = \log \left(\frac{|a-p|}{|a-q|}\right).
\end{equation}
If the ray $L^+$ is contained in the half-space $\mathcal{U}$, then $\dfunk(p,q) = 0$ and the above formula still holds
in the limit sense if we consider the point $a$ to be at infinity.

\medskip

For the general case, we will need two lemmas on the general tautological Finsler structure.

\begin{lemma}\label{lem.monotone}
 Let $\mathcal{U}_1$ and $\mathcal{U}_2$ be two convex domains in $\r^n$. If $F_1$ and $F_2$ are the Lagrangians 
 of the corresponding tautological structures, then 
 $$
  \mathcal{U}_1  \subset \mathcal{U}_2 \quad \Longleftrightarrow \quad  F_1(x,\xi) \geq F_2(x,\xi)
 $$
for all $(x,\xi) \in T\mathcal{U}$. 
\end{lemma}

\begin{proof} We have indeed
$$
 F_1(x,\xi) = \inf \{t>0 \tq \xi \in t (\mathcal{U}_1-x)\}  \geq  
\inf \{t>0 \tq \xi \in t (\mathcal{U}_2-x)\} =   F_2(x,\xi).
$$
\end{proof}
 We also have the following  \emph{reproducing formula} for the tautological Lagrangian.

\begin{lemma}
 Let $\mathcal{U}$  be a convex domain in $\r^n$ and $p,q \in \mathcal{U}$. 
 If  $q = p + t \xi$ for some $\xi \in \r^n$, and $t \geq 0$, then
 \begin{equation} \label{eq.repr}
   F_f(q,\xi) = \frac{F_f(p,\xi)}{1-t F_f(p,\xi)}.
\end{equation}
\end{lemma}

\begin{proof}  If $\xi = 0$, then there is nothing to prove. Assume $\xi \neq 0$ and  denote by $L_p^+$ and $L_q^+$ the rays 
in the direction $\xi$  starting at $p$ and  $q$. We obviously have $L_p^+ \subset \mathcal{U}$ if and only if 
$L_q^+ \subset \mathcal{U}$, but this means that 
$$
 F_f(p,\xi) = 0 \quad \Longleftrightarrow \quad F_f(q,\xi) = 0
$$
so Equation (\ref{eq.repr}) holds  in this case. If on the other hand the rays are not contained in $\mathcal{U}$, 
then we have
$$
  L_p^+ \cap \partial\mathcal{U} = L_q^+ \cap \partial\mathcal{U}.
$$
We denote by $a= a(x,\xi)$ this intersection point;  then by definition of $F$ we have
\begin{equation} \label{eq.apq}
 a = q  + \frac{\xi}{F_f(q,\xi)} = p + \frac{\xi}{F_f(p,\xi)} . 
\end{equation}
Since $q = p+ t \xi$, we have 
$$
  \frac{\xi}{F_f(q,\xi)} = \left( \frac{1}{F_f(p,\xi)} - t \right) \xi,
$$
from which  (\ref{eq.repr}) follows, since $\xi \neq 0$. 
\end{proof}

\bigskip

To complete the proof of Theorem \ref{th.tautdist}, we need to compute the  
tautological distance between two points $p$ and $q$ in a general convex domain 
$\mathcal{U} \in \r^n$. We first  compute the length of the affine segment $[p,q]$ which we parametrize as
$$
  \beta (t) = p + t \xi, \qquad \xi = \frac{y-x}{|y-x|}, \ 
$$
where $0 \leq t \leq |y-x|$. Let us denote as above by $a$ the point 
$L_p^+ \cap \partial\mathcal{U}$ where $L_p^+$ is the ray with origin $p$
in the direction $q$. Then 
$$
  F_f(p,\xi) = \frac{1}{|a-p|},
$$
and using (\ref{eq.repr})  we have  
$$
 F_f(\beta(t), \dot \beta (t)) = F_f(p+t\xi, \xi) =  \frac{F_f(p,\xi)}{1-t F_f(p,\xi)}
 = \frac{\frac{1}{|a-p|}}{1- \frac{t}{|a-p|}} = \frac{1}{|a-p| - t}.
$$
The length of $\beta$ is then
$$
 \ell (\beta) = \int_{0}^{|q-p|}  \frac{dt}{|a-p| - t} = \log \left(\frac{1}{|a-p| - |q-p|} \right)
 - \log \left(\frac{1}{|a-p|}\right).
$$
But $q \in [p,a]$, therefore $|a-p| - |q-p| = |a-q|$ and we finally have
$$
 \ell (\beta) =  \log \left(\frac{|a-p|}{|a-q|}\right).
$$
This proves that 
\begin{equation*} \label{taut.distc1}
  \dfunk(p,q) \leq  \log \left(\frac{|a-p|}{|a-q|}\right).
\end{equation*}
In fact we have equality. To see this, choose a supporting hyperplane for $\mathcal{U}$ at $a$,
and let $\mathcal{H}$ be the corresponding half-space containing $\mathcal{U}$
(recall that a  hyperplane in $\r^n$ is said to \emph{support} the convex set  $\mathcal{U}$
if it meets the closure of that set and  $\mathcal{U}$ is contained in one of the half-space 
bounded by that hyperplane). Using
Lemma \ref{lem.monotone} and Equation (\ref{taut.distb}), we obtain
\begin{equation*}\label{taut.distc}
 \dfunk(p,q) \geq  \log \left(\frac{|a-p|}{|a-q|}\right).
\end{equation*}
The proof of Theorem \ref{th.tautdist} follows now from the two previous  inequalities.
%  (\ref{taut.distc1})  and  (\ref{taut.distc}).
\end{proof}

\bigskip

\begin{remark}
\textbf{A.)} From Equation (\ref{eq.apq}), one sees that $a-p = \frac{\xi}{F_f(p,\xi)}$ and  $a-q = \frac{\xi}{F_f(q,\xi)}$. 
Therefore, the Funk  distance can also be written as
\begin{equation}\label{eq.distFFqp}
   \dfunk (p,q) = \log\left( \frac{F_f(q,q-p)}{F_f(p,q-p)}\right).
\end{equation}
\textbf{B.)} The proof of Theorem \ref{th.tautdist} given here is taken from \cite{WeakFinsler}. Another interesting 
proof is given in \cite{Yamada}.
\end{remark}

\begin{proposition}\label{prop.disttaut}
The tautological distance $\dfunk$ in a proper convex domain $\mathcal{U}$ satisfies the following properties
 \begin{enumerate}[a.)]
  \item The distance $ \dfunk$ is   projective, that is  for any point $z\in [p,q]$ we have
 $$
   \dfunk(p,q) =  \dfunk(p,z) +  \dfunk(z,p).
 $$
  \item $\dfunk$ is invariant under affine transformations.
  \item  $\dfunk$ is forward  complete.
  \item  $\dfunk$ is  is not backward complete.
\end{enumerate}
\end{proposition}

The proof is easy, see also \cite{PT2013b} for more on Funk geometry.

\begin{proposition}
The unit speed geodesic starting at $p\in \mathcal{U}$ in the direction $\xi \in T_p\mathcal{U}$ is the path
\begin{equation}\label{eq.FunkGeodesic1}
  \beta_{p, \xi} (s) =  p + \frac{(1-\mathrm{e}^{-s})}{F_f(p,\xi)}\cdot  \xi
\end{equation}
\end{proposition}
 
\begin{proof}
Let us define $\beta$ by (\ref{eq.FunkGeodesic1});  we have then
 $$ a-p = \frac{\xi}{F(p,\xi)}, \quad \mbox{and} \quad a- \beta (s)  = \frac{\mathrm{e}^{-s}}{F(p,\xi)}\cdot  \xi.$$
 Therefore
 $$
   \dfunk(p,\beta (s))  = \log \left(\frac{|a-p|}{|a-\beta (s)|} \right)
   =    \log \left(\frac{1}{\mathrm{e}^{-s}} \right) = s.
$$
\end{proof}

We can also prove the proposition using the fact that (\ref{eq.FunkGeodesic1}) parametrizes
an affine segment and is therefore a minimizer for the length. We then only need to check that
the curve has unit speed. Indeed we have
  $$
  \dot \beta (s) =  \frac{\mathrm{e}^{-s}}{F_f(p,\xi)}\cdot  \xi,
 $$
 and using Equation (\ref{eq.repr}), we thus obtain
\begin{eqnarray*}
  F_f(\beta (s), \dot \beta (s)) &=&  \frac{\mathrm{e}^{-s}}{F_f(p,\xi)} \cdot F\left(p+  \frac{(1-\mathrm{e}^{-s})}{F_f(p,\xi)} \cdot  \xi , \xi \right)
  \\ &=&
   \frac{\mathrm{e}^{-s}}{F_f(p,\xi)} \cdot \frac{F_f(p,\xi)}{\left(1-\frac{(1-\mathrm{e}^{-s})}{F_f(p,\xi)} F_f(p,\xi)\right)}
   \\ &=& 1.
\end{eqnarray*}

\medskip

The tautological Finsler structure $F_f$ in a  convex domain $\mathcal{U}$ is not reversible. 
We can thus define the \emph{reverse tautological Finsler structure $F_f^*$} to be the Finsler structure whose Lagrangian
is defined as 
$$
  F_f^*(p,\xi) = F_f(p,-\xi).
$$
We then have the following
\begin{proposition}
 The distance $\dfunk^*$ associated to the reverse tautological Finsler structure in a proper convex domain $\mathcal{U} \subset \r^n$ is given by 
 \begin{equation}\label{Rtaut.distd}
 \dfunk^*(p,q) = \dfunk (q,p) = \log \left(\frac{|b-q|}{|b-p|}\right).
\end{equation}
 where $b$ is the intersection of the ray starting at $q$ in the direction $p$ with $\partial \mathcal{U}$.
\end{proposition}

The proof is obvious. Observe in particular that $\dfunk^*$ is also a projective metric.

\section{The Hilbert metric}

\begin{definition}
  The \emph{Hilbert Finsler structure} $F_h$  in a convex domain $\mathcal{U}$ is the 
  arithmetic symmetrisation of the tautological Finsler structure:
  $$
    F_h(x, \xi) = \frac{1}{2}(  F_f(p,\xi) + F_f(p,-\xi)).
  $$
\end{definition}

\medskip

The Hilbert Finsler structure is reversible by construction. Its tangent unit ball at a point $p\in \mathcal{U}$  
is obtained from the tautological unit ball by the following procedure from convex geometry: first take the polar dual of $(\mathcal{U}-p)$, 
then symmetrize this convex set and finally take again the 
polar  dual of the result. This procedure is called the \emph{harmonic symmetrization} of $\mathcal{U}$ based at $p$, see \cite{Hsym} for the details.

\begin{example}\label{ex.Klein} 
Symmetrizing the metric  (\ref{taut.ball}) we obtain the Hilbert metric 
in the unit ball $\mathbb{B}^n$:
\begin{equation}\label{KleinMetric}
  F_h(x,\xi) =  \frac{\sqrt{(1-\|x\|^2)\|\xi\|^2 + \langle x,\xi \rangle^2}}{(1-\|x\|^2) }.
\end{equation}
Observe that this is a Riemannian metric. We shall prove later that it has constant sectional curvature $\K = -1$.
This metric is the Klein model\index{Klein model} for hyperbolic geometry. 
\end{example}

Using the fact that the affine segment joining two points $p$ and $q$ in $\mathcal{U}$ has minimal length  for both the Funk metric $F_f$ and its reverse $F_f^*$, it is easy to prove the following

 \begin{proposition}
The Hilbert distance $\dhilb$ in a convex domain $\mathcal{U} \subset \r^n$ 
is obtained by symmetrizing the Funk distance (\ref {taut.distd}) in that domain:
\begin{equation}\label{hilbert.distd}
 \dhilb(p,q) = \frac{1}{2}(\dfunk(p,q) + \dfunk(q,p))
  = \frac{1}{2}\log \left(\frac{|a-p|}{|a-q|} \cdot \frac{|b-q|}{|b-p|}\right).
\end{equation}
\end{proposition}

\vspace{0.25cm}

\hspace{3.6cm}\scalebox{0.9} % Change this value to rescale the drawing.
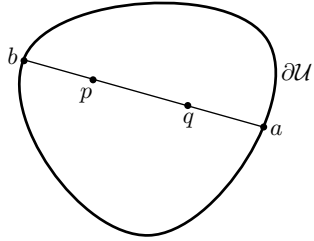
{
\begin{pspicture}(0,-1.9007623)(4.55775,1.900762)
\psbezier[linewidth=0.04](0.44074503,0.80477965)(0.88149005,1.702412)(3.50374,1.8807621)(4.020745,1.0247797)(4.53775,0.16879725)(3.2604187,-1.8296787)(2.260745,-1.8552203)(1.2610712,-1.8807621)(0.0,-0.0928527)(0.44074503,0.80477965)
\psline[linewidth=0.02cm](0.400745,0.74477965)(3.960745,-0.25522035)
\psdots[dotsize=0.103999995](0.42074502,0.72477967)
\psdots[dotsize=0.103999995](3.960745,-0.25522035)
\psdots[dotsize=0.103999995](1.440745,0.44477966)
\psdots[dotsize=0.103999995](2.840745,0.064779654)
\usefont{T1}{ptm}{m}{n}
\rput(1.34,0.22){$p$} % p
\usefont{T1}{ptm}{m}{n}
\rput(2.85,-0.15){$q$} % q
\usefont{T1}{ptm}{m}{n}
\rput(0.255,0.8){$b$}  % b
\usefont{T1}{ptm}{m}{n}
\rput(4.15,-0.29){$a$}  % a
\rput(4.45,0.56){$\partial \mathcal{U}$}
\end{pspicture} 
}

\vspace{0.3cm}

Using (\ref{eq.distFFqp}), one can also write the Hilbert distance  as
\begin{equation}\label{hilbert.distd2}
 \dhilb(p,q)    = \frac{1}{2}\log \left( \frac{F_f(q,q-p)}{F_f(p,q-p)} \cdot  \frac{F_f(p,p-q)}{F_f(q,p-q)}\right).
\end{equation}

\bigskip

We also have the following properties:
\begin{proposition}\label{prop.disthilb}
The  Hilbert distance $\dhilb$ in a proper convex domain $\mathcal{U}$ satisfies the following 
 \begin{enumerate}[a.)]
  \item The distance $ \dhilb$ is   projective:  for any point $z\in [p,q]$ we have
 $
   \dhilb(p,q) =  \dhilb(p,z) +  \dhilb(z,p).
 $
  \item $\dhilb$ is invariant under projective transformations.
  \item  $\dhilb$ is  (both forward and backward) complete.
\end{enumerate}
\end{proposition}

\medskip

The proof is elementary. For an introduction to Hilbert geometry,  we refer to sections 28 and 50 in  \cite{Busemann1955} and section 18 in  and  \cite{Busemann1955}.
We now describe the geodesics in Hilbert geometry. 

\begin{proposition}
The unit speed geodesic starting at $p\in \mathcal{U}$ in the direction $\xi \in T_p\mathcal{U}$ is the path
\begin{equation}\label{eq.HilbGeodesic1}
  \beta_{p, \xi} (s) =  p + \varphi(s)\cdot  \xi,
\end{equation}
where $\varphi$ is given by
\begin{equation}\label{eq.HilbGeodesic1a}
 \varphi(s) = \frac{(\mathrm{e}^s - \mathrm{e}^{-s})}{F_f(p,\xi)\mathrm{e}^s + F_f(p,-\xi)\mathrm{e}^{-s}}.
\end{equation}
\end{proposition}

\medskip

\begin{proof}
 To simplify notation we write $F = F_f(p,\xi)$ and $F^* = F_f(p,-\xi)$. 
 From (\ref{eq.apq}) we have $a-p = \frac{\xi}{F}$
 and $b-q = - \frac{\xi}{F^*}$, therefore
 $$
   \varrho_h(p,\beta (s))  = \frac{1}{2}\log \left(\frac{|a-p|}{|a-\beta (s)|} \cdot \frac{|b-\beta (s)|}{|b-p|}\right)
   =  \frac{1}{2}\log \left( \frac{1+F^*\cdot \varphi(s)}{1-F \cdot \varphi (s)} \right).
 $$
From (\ref{eq.HilbGeodesic1a}), we have 
  $$
  \frac{1+F^*\cdot\varphi(s)}{1-F\cdot \varphi (s)} =  \frac{(F\cdot \mathrm{e}^s + F^*\cdot \mathrm{e}^{-s})  +  
  (F^*\cdot \mathrm{e}^s - F^*\cdot \mathrm{e}^{-s})}{(F\cdot \mathrm{e}^s + F^*\cdot \mathrm{e}^{-s})  -  (F\cdot \mathrm{e}^s - F\cdot \mathrm{e}^{-s})} 
 = \frac{(F + F^*)\cdot e^{s}}{(F+ F^*)\cdot \mathrm{e}^{-s}}  = e^{2s},
 $$
 and we conclude that 
 $$
     \dhilb(p,\beta (s))  = s.
 $$
\end{proof}

\begin{corollary}
 The metric balls in a Hilbert geometry are convex sets.
\end{corollary}

\begin{proof}
 We see from the previous proposition that the ball of radius $r$ around the point $p \in \mathcal{U}$ is the set of
 points $z\in  \mathcal{U}$ such that 
 $$
  \mathrm{e}^{r}F_f(p,z-p)
    + \mathrm{e}^{-r}F_f(p,p-z)  \leq {(\mathrm{e}^r - \mathrm{e}^{-r})},
 $$
which is a convex set.
\end{proof}

Note that another proof of this proposition is given in  \cite{PT2013b}.

\section{The fundamental tensor}

Our goal in this section is to write down and study an ordinary differential equation for 
the geodesics in a Finsler manifold $(M,F)$. To this aim we need to assume the following condition:

\begin{definition}\label{def.Fstronglyconvex}
 A Finsler metric $F$ on a smooth manifold $M$ is said to be \emph{strongly convex}\index{strongly convex Finsler metric} if
 it is smooth  on the slit tangent bundle $TM^0$ and if the  \emph{vertical Hessian} of $F^2$ at a point $(p,\xi) \in T^0M$  
\begin{equation}\label{fund.tensor0}
  \g_{p,\xi}(\eta_1,\eta_2) = \frac{1}{2}\left.\frac{\partial^2}{\partial u_1\partial u_2}\right|_{u_1=u_2=0} F^2(p, \xi +u_1\eta_1+u_2\eta_2)
\end{equation}
is positive definite.  The bilinear form (\ref{fund.tensor0}) is called the \emph{fundamental tensor}\index{fundamental tensor (of a Finsler metric)} of the Finsler metric.
\end{definition}

This condition is called the \emph{LegendreÐClebsch condition}\index{LegendreÐClebsch condition}   in the calculus of variations.  
Geometrically it means that the indicatrix at any point  is a hypersurface of strictly positive Gaussian curvature. 
 
 \medskip
 
Classical Finsler geometry has sometimes the  reputation of being an ``impenetrable forest of tensors\footnote{This comment on the subject goes back to the paper \cite{Busemann1950} by Busemann. The first sentence
 of this nice paper is  \emph{``The term \emph{Finsler space} evokes in most mathematicians the picture of an impenetrable forest whose entire vegetation consists of tensors.''} A quick glance at \cite{Rund59} will probably convince the reader.}'',  and 
we shall   need to venture a few steps into this wilderness.
It will be convenient to work with local coordinates on $TM$,  more precisely, if $U \subset M$ is the domain 
of some coordinate system $x^1, x^2, \dots, x^n$, then any vector $\xi \in TU$ can be written  as $\xi = y^i\frac{\partial}{\partial y^i}$. The 
$2n$ functions on $TU$ given by $x^1,  \dots, x^n, y^1,  \dots, y^n$  are called \emph{natural coordinates} on $TU$.
The restriction of the Lagrangian on $TU$ is thus given by a function of $2n$ variables $F(x,y)$. 
The Legendre-Clebsch condition states that $F(x,y)$ is smooth (although  $C^3$ would suffice) on $\{ y \neq 0\}$ and that the matrix given by
\begin{equation}\label{fund.tensor}
 g_{ij}(x,y) = \frac{1}{2}\frac{\partial^2 F^2}{\partial y^i\partial y^j}(x,y)
\end{equation}
is positive definite   for any $y\neq 0$.
The fundamental tensor   shares some formal properties with a Riemannian metric, but it is important to remember that it is not defined on $U$ (nor on the manifold $M$), but on the slit tangent bundle $T^0U$. 

\medskip

Manipulating tensors in Finsler geometry needs to be done with care. In general a \emph{tensor} on a Finsler manifold $(M,F)$ is a field of multilinear maps on $TM$
which smoothly depends on a point $(x,y) \in TM^0$ (and not only on a point $x\in M$ as in Riemannian geometry). In the case of the fundamental
tensor, note that\footnote{We use the summation convention: terms with repeated indices are summed from $1$ to $n$.}
$$
 \g_y(u,v) = \g_{x,y}(u,v) = g_{ij}(x,y)u^iv^j
$$
where $u = u_i(x)\frac{\partial}{\partial x^i}$ and  $v = v_j(x)\frac{\partial}{\partial x^j}$ are elements in $T_xM$.

\medskip

Given a point $(x,y) \in TM^0$, we have a canonical element in $T_xM$, namely the vector $y$ itself, and we can evaluate a given tensor on
that canonical vector. In particular we have the following basic fact.

\begin{lemma}\label{lem.gnonhol}
  The Lagrangian and the fundamental tensor are related by
 $$
  F(x,y) = \sqrt{\g_y(y,y)}= \sqrt{g_{ij}(x,y)y^iy^j}.
 $$
\end{lemma}

\begin{proof}
Recall that if  $f : \r^n\setminus \{0\}\to \r$ is a smooth positively homogenous function of degree $r$ on $\r^n$, that is,  $f(\lambda y) = \lambda^rf(y)$ for $\lambda \geq 0$, then its 
partial derivatives  are positively homogenous functions of degree $(r-1)$ and
$
  r\cdot f(y) = y^i\frac{\partial f}{\partial y^i}
$
(see \cite{PT2013a}). Applying this fact twice to the function $y \mapsto \frac{1}{2}F^2(x,y)$  proves  the 
Lemma.
\end{proof}

\medskip

Observe that the same argument also shows that the  fundamental tensor $g_{ij}(x,y)$ is $0$-homogenous with respect to $y$.
This type of argument  plays a central role in Finsler geometry and anyone venturing into the subject will soon 
become an expert in recognizing
how homogeneity is being used (sometimes in a hidden way) in calculations, or she will quit the
subject.

\section{Geodesics and the exponential map}

We now consider a curve $\beta : [a,b] \to M$ of class $C^1$ in the Finsler manifold $(M,F)$. 
Recall that its \emph{length} is defined as
\begin{equation}\label{eq.length}
  \ell (\beta) = \int_a^b F(\beta (s), \dot \beta (s)) ds.
\end{equation}
We also define the \emph{energy}\index{energy of a curve} of the curve $\beta$ by
\begin{equation}\label{eq.energy}
  E(\beta) = \int_a^b F^2(\beta (s), \dot \beta (s)) ds.
\end{equation}
The following basic inequality between these functionals holds:

\begin{lemma}
For any curve $\beta : [a,b] \to M$ of class $C^1$,  we have
$$
 \ell(\beta)^2 \leq (b-a) \, E(\beta),
$$
with equality if and only if $\beta$ has constant speed, i.e. $t \mapsto F(\beta (s), \dot \beta (s))$ is constant.
\end{lemma}

\begin{proof}
Let us denote by $f(s)= F(\beta (s), \dot \beta (s))$ the speed of the curve. We have by the Cauchy-Schwarz inequality
$$
  \ell (\beta) = \int_a^b 1 \cdot f(s)ds \leq \left( \int_a^b 1^2 ds\right)^{1/2}  \left( \int_a^b v(t)^2 ds\right)^{1/2} 
  = \sqrt{(b-a)}\sqrt{E(\beta)}.
$$
Furthermore, the equality holds if and only if $f(s)$ and $1$ are collinear as element of $L^2(a,b)$, that is, if
and only if the speed $f(s)$ is constant.
\end{proof}

\medskip

\begin{corollary}
 The curve $\beta : [a,b] \to M$ is a minimal curve for the energy  if and only if it minimizes the length and has constant speed.
\end{corollary}

\medskip

\begin{definition}
 The curve $\beta : [a,b] \to M$ is a \emph{geodesic}\index{geodesic} if it is a critical point for the energy functional.
\end{definition}

Sometimes the terminology varies,  and the term \emph{geodesic} also designates a critical or a minimal curve for the length 
functional. Note that our notion imposes the restriction that $\beta$ has constant speed. From the point of view of calculations,
this is an advantage since the length is invariant under forward reparametrization.

\medskip

We now seek to derive the equation satisfied by the geodesics. By the classical theory of the  calculus of variations,
a curve $s \mapsto \beta (s)$ in some coordinate domain $U \subset M$ is a
critical points of the energy  functional  (\ref{eq.energy}) if and only if the following \emph{Euler-Lagrange equations}\index{Euler-Lagrange equations}
\begin{equation}\label{EulerLagrange}
 \frac{d}{ds}\frac{\partial F^2}{\partial y^{\mu}} = \frac{\partial F^2}{\partial x^{\mu}} \quad (\mu = 1,\dots, n)
\end{equation}
hold, where $ (x(s), y(s)) = (\beta (s),\dot \beta (s))$.

\medskip

Using the fundamental tensor, one writes $F^2(x,y) = g_{ij}(x,y)y^iy^j$, therefore
\begin{equation}\label{eq.geoda}
 \frac{\partial F^2}{\partial x^{\mu}} = \frac{\partial g_{ij}}{\partial x^{\mu}} y^iy^j
\end{equation}
and 
$$
  \frac{\partial F^2}{\partial y^{\mu}} = \frac{\partial g_{ij}}{\partial y^{\mu}} y^iy^j + g_{i\mu}y^i + g_{j\mu}y^j.
$$
Observe that the sum $g_{i\mu}y^i$ and $g_{j\mu}y^j$ coincide. Using the homogeneity of $F^2$
in $y$ and  the fact that $F$ is of class $C^3$ on $y\neq 0$, we obtain
$$
  \frac{\partial g_{ij}}{\partial y^{\mu}} y^i =  \frac{\partial^3 F^2}{\partial y^{\mu}\partial y^i \partial y^j}\cdot  y^i
  =  \frac{\partial^3 F^2}{\partial y^i \partial y^{\mu}\partial y^j}\cdot  y^i = 0.
$$
Therefore
\begin{equation}\label{eq.geodb}
 \frac{\partial F^2}{\partial y^{\mu}} = 2 g_{i\mu}y^i.
\end{equation}
Differentiating this expression with respect to $s$, we get    
$$
 \frac{d}{ds}\frac{\partial F^2}{\partial y^{\mu}} = 2\frac{\partial g_{i\mu}}{\partial x^j}\cdot y^i \cdot\frac{d x^j}{ds}
 + 2 g_{i\mu}\cdot \frac{dy^i}{ds} + 2 \frac{\partial g_{i\mu}}{\partial y^j}\cdot y^i \cdot\frac{d y^j}{ds}.
$$
We have as before $\frac{\partial g_{i\mu}}{\partial y^j}\cdot y^i =0$ (because $ g_{i\mu}$ is $0$-homogenous),
therefore using $\dot x^i = y^i$ we obtain
\begin{equation}\label{eq.geodc}
  \frac{d}{ds}\frac{\partial F^2}{\partial y^{\mu}} =  2\frac{\partial g_{i\mu}}{\partial x^j}\cdot y^iy^j + 2 g_{i\mu}\dot y^i.
\end{equation}
From the equations (\ref{EulerLagrange}) (\ref{eq.geodb}) and (\ref{eq.geodc}), we obtain for $\mu = 1, \dots, n$
$$
 \sum_i g_{i\mu}\cdot \dot y^i = \frac{1}{2}\sum_{i,j}\left(  \frac{\partial g_{ij}}{\partial x^{\mu}}
 - 2 \frac{\partial g_{i\mu}}{\partial x^j} \right) y^iy^j.
$$
Multiplying this identity by $g^{k\mu}(x,y)$, where $g^{k\mu}g_{\mu i}= \delta^k_i$, and summing over $\mu$, one gets
\begin{equation}\label{eq.geodd}
 \dot y^k = \frac{1}{2}\sum_{i,j, \mu} g^{k\mu}\left(  \frac{\partial g_{ij}}{\partial x^{\mu}}
 - 2 \frac{\partial g_{i\mu}}{\partial x^j} \right) y^iy^j.
\end{equation}  

\begin{proposition}
 The $C^2$ curve $\beta (s)  \in (M,F)$ is a geodesic for the Finsler metric $F$ if only if
 in local coordinates, we have 
 \begin{equation}\label{eq.geod1}
 \ddot x^k + \gamma^k_{ij}\, \dot x^i\dot x^j = 0,
\end{equation}
where $\beta (s) = (x^1(s), \dots, x^n(s))$ is a local coordinate expression for the curve and 
$$
 \gamma^k_{ij}(x,y)  =  \frac{1}{2}g^{k\mu}\left(\frac{\partial g_{i\mu}}{\partial x^j} + \frac{\partial  g_{j\mu}}{\partial x^i} - \frac{\partial g_{ij}}{\partial x^{\mu}}\right),
$$
are  the \emph{formal Christoffel symbols}\index{formal Christoffel symbols} of $F$.
\end{proposition}

\begin{proof}
 A direct calculation shows that Equation (\ref{eq.geod1}) is equivalent to (\ref{eq.geodd}) with $y^k = \dot x^k$. 
 \end{proof}
 
\medskip

Another way to write the geodesic equation is to introduce the functions
\begin{eqnarray*}\label{eq.Spray1}
 G^k(x,y)  &=&     \frac{1}{2}\gamma^k_{ij}(x,y)\, y^i y^j
 \\ &=&
   \frac{1}{4}g^{k\mu}\left(  2 \frac{\partial g_{i\mu}}{\partial x^j}  -  \frac{\partial g_{ij}}{\partial x^{\mu}}\right) y^iy^j
 \\ &=&    
   \frac{1}{2}g^{k\mu}\left( \frac{\partial^2 F}{\partial y^{\mu}\partial x^j} y^j - \frac{\partial F}{\partial x^{\mu}} \right),
\end{eqnarray*}
so the geodesic equation can be written as
\begin{equation} \label{eq.geod2}
 \dot y^k + 2G^k(x,y) = 0, \quad \dot x^k = y^k, \qquad(k = 1, \dots n).
\end{equation}
The vector field on $TU$ defined by 
\begin{equation}\label{eq.Spray2}
 G = y^k \frac{\partial}{\partial x^k} - 2G^k(x,y) \frac{\partial}{\partial y^k}
\end{equation}
is in fact independent on the choice of the coordinates $x^j$ on $U$ and is therefore globally defined on the 
tangent bundle $TM$. This vector field is called the \emph{spray}\index{spray} of the Finsler manifold. A significant 
part of Finsler geometry is contained in the behavior of its spray (see \cite{Shen2001a}). Observe that a curve $\beta(s) \in M$ is geodesic
if and only if its lift $(\beta (s), \dot \beta (s)) \in TM$ is an integral curve for the spray $G$.

\medskip

Observe the rather surprising analogy between the equation (\ref{eq.geod1}) and  the Riemannian geodesic equation. 
This is due to the fact that the $0$-homogeneity of $g_{i,j}(x,y)$ in $y$  implies that all the non-Riemannian terms in the 
calculation of the Euler-Lagrange equation end up vanishing.
The important difference between the Finslerian and the Riemannian cases 
lies in the fact that the formal Christoffel symbols $\gamma^k_{ij}$ are not functions of
the coordinates $x^i$ only. Equivalently, the spray coefficients are  generally not quadratic polynomials in the
coordinates $y^i$. In fact it is known that the spray coefficients $G^k(x,y)$ of  a Finsler metric $F$ are quadratic 
polynomials in the coordinates $y^i$ if and only if $F$ is Berwald.

\medskip   

On a smooth manifold $M$ with a strongly convex  Finsler metric, one can define an exponential map as it is done in Riemannian
geometry.  Because the coefficients $G^k(x,y)$ of the geodesic equation are Lipschitz continuous, for any  
point $p\in M$ and any  vector $\xi \in T_pM$, there is locally a unique solution to the geodesics  equation
$$
 s \mapsto \sigma_{\xi}(s) \in M, \qquad  -\epsilon < s < \epsilon
$$
with initial conditions $\sigma_{\xi}(0)= p$, $\dot \sigma_{\xi}(0) = \xi$. Observe 
that 
$$
  \sigma_{\lambda\xi}(s) = \sigma_{\xi}(\lambda s),
$$ 
so if $\xi$ is small enough, then $\sigma_{\xi}(1)$ is well defined and we denote it by
$$
 \exp_p(\xi) = \sigma_{\xi}(1).
$$
We then have the following 

\begin{theorem}
The set $\mathcal{O}_p$ of vectors $\xi \in T_pM$ for which $\exp_p(\xi)$ is defined is a neighborhood of $0\in T_pM$. The map 
$\exp_p : \mathcal{O}_p \to M$ is of class $C^1$ in the interior of $\mathcal{O}_p$ and its differential at $0$ is the identity. In particular, 
the exponential is a diffeomorphism from a neighborhood of $0 \in T_pM$ to a neighborhood of $p$ in $M$.

If $(M,F)$ is  connected and forward complete, then $\exp_p$ is defined on all of $T_pM$ and is a surjective map
$$
 \exp_p : T_pM \to M.
$$
\end{theorem}

The coefficients $G^k(x,y)$ in the geodesic equation are in general not $C^1$ at $y=0$, therefore the proof of this theorem
is more delicate than its Riemannian counterpart. See \cite{BCS} for the details.

\medskip

Finally note that  $\exp_p$ has no reason to be of class $C^2$. In fact a result of Akbar-Zadeh  states that $\exp_p$
is a $C^2$ map near the origin if and only if $(M,F)$ is Berwald.  

\section{Projectively flat Finsler metrics}

\begin{definition}  \textbf{A.)}  A Finsler structure $F$ on a convex set $\mathcal{U} \subset \r^n$  is \emph{projectively flat} if every affine segment $[p,q]\subset 
\mathcal{U}$ can be parametrized as a geodesic.   

\smallskip

 \textbf{B.)}   A Finsler manifold $(M,F)$ is \emph{locally projectively flat}  if each point admits a neighborhood 
 that is isometric to a convex region in $\r^n$ with a  projectively Finsler structure.
\index{projectively flat Finsler manifold}
\end{definition}

\begin{examples}\label{ex.Fplat1}
Basic examples   are:
\begin{enumerate}[a.)]
  \item Minkowski metrics are obviously projectively flat.
  \item The Funk and the Hilbert metrics are  projectively flat Finsler metrics. In particular the Klein metric (\ref{KleinMetric})
  in $\mathbb{B}^n$ is a projectively flat Riemannian metric.
  \item The canonical metric on the sphere $S^n \subset \r^{n+1}$ is locally projectively flat.
  Indeed,  the central projection of  the half-sphere $S^n \cap \{ x_n < 0\}$ on the hyperplane $\{ x_n = -1\}$ with center the origin maps the great circles in $S^n$ on affine lines on  that hyperplane (in cartography, this map is called the \emph{gnomonic  representation of the sphere}). In formulas, the map $\r^n \to S^n$ is given by $x\mapsto \frac{(x,-1)}{\sqrt{1+\|x\|^2}}$, and the metric can be written  as
 \begin{equation}\label{SphericalMetric}
  F_x(x,\xi) =  \frac{\sqrt{(1+\|x\|^2)\|\xi\|^2 - \langle x,\xi \rangle^2}}{(1+\|x\|^2) }.
\end{equation}
This is the projective model for the spherical metric.
 \end{enumerate}
\end{examples}

A classic result, due to E. Beltrami, says that a Riemannian manifold is  locally projectively flat if and only if it has constant sectional curvature,  see Theorem \ref{th.Beltrami} below. 
In Finsler geometry, there are more examples and the Finlser version of Hilbert's  $\mathrm{IV}^{\mathrm{th}}$ problem is to determine and study all conformally flat (complete) Finsler metrics in a convex domain.

\medskip

A first result on projectively flat Finsler metrics is given in the next Proposition:

\begin{proposition}
A smooth and strongly convex Finsler metric $F$ on a convex domain $\mathcal{U}  \subset \r^n$ is projectively flat if only if its spray coefficients satisfy 
 $$
   G^k(x,y) = P(x,y) \cdot y^k
 $$
 for some scalar function $P : T\mathcal{U} \to \r$.
\end{proposition}

\begin{proof} The proof is standard, see e.g. \cite{ChernShen}.  Suppose that the affine segments are geodesics.
This means that 
$$
 \beta (s) = p + \varphi(s) \xi
$$  
satisfy the geodesic equation (\ref{eq.geod2}) for any $p \in \mathcal{U}$ , $\xi \in \mathcal{U}$ and some (unknown) function $\varphi(s)$. 
We  have along $\beta$
$$
 y^k = \dot x^k = \dot \varphi(s) \xi^k, \quad  \dot y^k =  \ddot \varphi(s) \xi^k,
$$ 
therefore Equation (\ref{eq.geod2}) can be written as
$$
 \ddot \varphi(s) \xi^k + 2 \dot \varphi^2(s)  G^k(\beta (s),\xi) = 0,
$$
which implies that  $G^k(x,y) = P(x,y) \cdot y^k$ with
\begin{equation}\label{eq.P1}
 P(x,\xi) =  - \frac{ \ddot \varphi (0) }{2 \dot \varphi (0)^2}.
\end{equation}
\end{proof}
 
The function $P(x,\xi)$ in the previous Proposition is called the \emph{projective factor}. It can be computed from (\ref{eq.P1}) if the 
geodesics are explicitly known. For instance we have the

\begin{proposition}
 The projective factor of the tautological Finsler structure $F_f$  in the convex domain $\mathcal{U}$
 is given by
 $$
   P_f(x,y) = \frac{1}{2} F_f(x,y) .
 $$
For the Hilbert Finsler structure $F_h$, we have 
 $$
   P_h(x,y) = \frac{1}{2} (F_f(x,y) - F_f(x,-y)).
 $$
\end{proposition}

\begin{proof} 
The geodesics for $F_f$ are given by  $\beta (s) = p + \varphi(s) y$
with
$$ 
 \varphi(s) = \frac{1}{F_f(p, \xi)}(1-\mathrm{e}^{-s}).
$$
Therefore
$$
 P_f(x,y) = - \frac{ \ddot \varphi (0) }{2 \dot \varphi (0)^2} 
 =  \frac{1}{2} F_f(p, \xi).
$$ 
The geodesics for the Hilbert metric $F_h(x,y) = \frac{1}{2}(F_f(x,y)+F_f(x,-y))$ are the curves
$\beta (s) = p + \varphi(s) y$ with
\begin{equation*}%\label{eq.HilbGeodesic1a}
 \varphi(s) = \frac{(\mathrm{e}^s - \mathrm{e}^{-s})}{F_f(p,\xi)\mathrm{e}^s + F_f(p,-\xi)\mathrm{e}^{-s}}.
\end{equation*}
The derivative of this function is
$$
 \dot \varphi(s) = \frac{2(F_f(p,\xi)+ F_f(p,-\xi))}{(F_f(p,\xi)\mathrm{e}^s + F_f(p,-\xi)\mathrm{e}^{-s})^2}, 
$$
and
$$
 \ddot \varphi(s) = -4 \, \frac{(F_f(p,\xi)+ F_f(p,-\xi))(F_f(p,\xi)\mathrm{e}^s - F_f(p,-\xi)\mathrm{e}^{-s})}{(F_f(p,\xi)\mathrm{e}^s + F_f(p,-\xi)\mathrm{e}^{-s})^3}.
$$
Therefore
$$
 P_h(x,y) =  - \frac{ \ddot \varphi (0) }{2 \dot \varphi (0)^2} 
 =  \frac{1}{2} (F_f(x,y) - F_f(x,-y)).
$$
\end{proof}

\medskip

It is clear that the distance associated to a projectively flat metric is projective in the sense of Definition
2.1 in \cite{PT2013a}.   It is also clear that the sum of two projective (weak) distances is again projective.
This suggests that one can write down a linear condition on $F$ that is equivalent to projective flatness.
This is the content of the next proposition which goes back to the work of  G. Hamel,
\cite{Hamel}.

\begin{proposition}\label{prop.hamel}
 Let $F : T\mathcal{U} \to \r$ be a smooth  Finsler metric on the convex domain $\mathcal{U}$. 
 The following conditions are
 equivalent.
 \begin{enumerate}[(a.)]
  \item $F$ is projective flat.
  \item $\ds y^k \frac{\partial^2 F}{\partial x^k\partial y^m} - \frac{\partial F}{\partial x^m} = 0$, \ for $1 \leq m \leq n$.
  \item  $\ds \frac{\partial^2 F}{\partial x^j\partial y^m} =  \frac{\partial^2 F}{\partial y^j\partial x^m} $, \  for $1 \leq j,m \leq n$.
  \item $\ds    y^k\frac{\partial^2 F}{\partial x^m\partial y^k} = y^k\frac{\partial^2 F}{\partial x^k\partial y^m}.$
\end{enumerate}
\end{proposition}

\begin{proof}
The Euler-Lagrange equation for the length functional (\ref{eq.length}) is
$$
 0 = \frac{\partial F}{\partial x^m} - \frac{d}{dt}\frac{\partial F}{\partial y^m} 
 = \frac{\partial F}{\partial x^m} - \frac{\partial^2 F}{\partial x^k\partial y^m} \dot x^k - \frac{\partial^2 F}{\partial y^k\partial y^m} \dot y^k,
$$
and this can be written as
$$
\frac{\partial^2 F}{\partial y^k\partial y^m} \dot y^k
=
\frac{\partial F}{\partial x^m} - \frac{\partial^2 F}{\partial x^k\partial y^m} y^k.
$$
Now $F$ is projectively flat if and only if $x(t) = p + t \xi$ is a solution (recall that the length is invariant under 
reparametrization), which is
equivalent to $\dot y^k = 0$. The equivalence (a) $\Leftrightarrow$ (b) follows.

\smallskip

To prove (b) $\Rightarrow$ (c), we differentiate (b) with respect to $y^j$. This gives
$$
 A_{jm} =  y^k\frac{\partial^3 F}{\partial y^j\partial x^k\partial y^m} +
 \frac{\partial^2 F}{\partial x^j\partial y^m} - \frac{\partial^2 F}{\partial y^j\partial x^m} = 0,
$$
and thus
$$
  0 = \frac{1}{2}(A_{jm} - A_{mj}) = \frac{\partial^2 F}{\partial x^j\partial y^m} - \frac{\partial^2 F}{\partial y^j\partial x^m} ,
$$
which is equivalent to (c).

\smallskip

(c) $\Rightarrow$ (d) is obvious.

\smallskip

Finally,  to prove (d) $\Rightarrow$ (b), we use  the  homogeneity of $F$ in $y$. We have $F(x,y) =  y^k\frac{\partial F}{\partial y^k}$, 
therefore condition (d) implies
$$
 \frac{\partial F}{\partial x^m} =  y^k\frac{\partial^2 F}{\partial x^m\partial y^k} = y^k\frac{\partial^2 F}{\partial x^k\partial y^m},
$$
which is equivalent to (b).
\end{proof}

\begin{remark}
In \cite{Rapcsak}  A. Rapcsk gave a generalization of the previous conditions for the case of a pair of
projectively equivalent Finsler metrics, that is,  a pair of metrics having the same geodesics up to reparametrization.
\end{remark}

\smallskip

\begin{example}
 We know that the tautological (Funk) Finsler metric $F_f$ in a convex domain $\mathcal{U}$ is projectively flat. Using the previous proposition we have
 more examples:
 \begin{enumerate}[i)]
  \item The reverse Funk metric $F_f^*(x,\xi) = F_f(x, -\xi)$,
  \item The Hilbert metric $F_h(x, \xi) = \frac{1}{2}(F_f(x, -\xi)+F_f(x, -\xi))$,
  \item The metric $F_f(x, \xi) + F_0(\xi)$, where $F_0$ is an arbitrary (constant) Minkowski norm,
\end{enumerate}
are projectively flat. More generally if $F_1,F_2$ are projectively flat, then so is the sum $F_1+F_2$. Assuming either 
$F_1$ or $F_2$ to be (forward) complete,
the sum is also a (forward) complete solution to the Hilbert   $\mathrm{IV}^{\mathrm{th}}$   problem.
\end{example}
 
  %\medskip
 
The following result gives a general formula computing the projective factor of a  
projectively flat Finsler metric:

\begin{lemma}\label{lem.PF}
 Let $F(x,y)$ be a smooth projectively flat Finsler metric on some domain in $\r^n$. Then 
 the following equations hold
\begin{equation}\label{eq.2FP}
   2F P = y^k\frac{\partial F}{\partial x^k}
\end{equation}
 and
\begin{equation}\label{eq.2FPP}
  \frac{\partial F}{\partial x^m} = P  \frac{\partial F}{\partial y^m} + F  \frac{\partial P}{\partial y^m},
\end{equation}
 where $P(x,y)$ is the projective factor.
\end{lemma}

\begin{proof}
Along a geodesic $(x(s),y(s))$, we have 
$$
  \dot y^k = -2 G^k(x,y) = -2 P(x,y) y^k, \qquad  \dot x^k = y^k.
$$
Since the Lagrangian $F(x(s),y(s))$  is constant in $s$, we obtain the first equation
$$
 0 = \frac{dF}{ds} = \frac{\partial F}{\partial x^k} y^k + \frac{\partial F}{\partial y^k} \dot y^k
 =  \frac{\partial F}{\partial x^k} y^k  -2 P \frac{\partial F}{\partial y^k} y^k
 = \frac{\partial F}{\partial x^k} y^k  -2 P  F.
$$
Differentiating this equation and using (b) in Proposition \ref{prop.hamel}, we obtain
$$
2\left( P  \frac{\partial F}{\partial y^m} + F  \frac{\partial P}{\partial y^m}\right)
=   \frac{\partial }{\partial y^m} \left( y^k \frac{\partial F}{\partial x^k} \right) 
=    \frac{\partial F}{\partial x^m}  + y^k \frac{\partial^2 F}{\partial x^k\partial y^m} 
= 2 \frac{\partial F}{\partial x^m}. 
$$
\end{proof}

A first consequence is the following description of Minkowski metrics:

\begin{corollary}\label{cor.PnulMink}
A  strongly convex projectively flat Finsler metric on some domain in $\r^n$ is the restriction
of a Minkowski metric if and only if the associated  projective factor $P(x,y)$ vanishes  identically.
\end{corollary}

\begin{proof}
 Suppose $F(x,y)$ is locally Minkowski, then $\frac{\partial F}{\partial x^m} = 0$ and therefore the spray coefficient satisfies
 $P(x,y) y^k = G^k(x,y) = 0$.
 Conversely, if $P(x,y) \equiv 0$, then the second equation in the lemma implies that $\frac{\partial F}{\partial x^m} = 0$.
\end{proof}

Another consequence is the following result about the projective factor of a reversible Finsler metric.

\begin{corollary}\label{cor.Preversible}
 Let $F(x,y)$ be a strongly convex projectively flat Finsler metric. 
 Suppose $F$ is reversible, then its projective factor satisfies
 $$
   P(x,-y) = - P(x,y).
 $$
\end{corollary}
 
\begin{proof}
 This is obvious from the first equation in Lemma \ref{lem.PF}.
\end{proof}

\section{The Hilbert form}
 
 We now introduce the  Hilbert $1$-form of a Finsler manifold and show its relation 
 with Hamel's condition:
\begin{definition}  
 The \emph{Hilbert $1$-form}\index{Hilbert form} on a smooth Finsler manifold $(M,F)$ is the $1$-form on $TM^0$ defined in natural coordinates as
 $$
   \omega   = \frac{\partial F}{\partial y^j} dx^j = g_{ij}(x,y)y^idx^j.
 $$
\end{definition}

From the homogeneity of $F$, we have $F(x,\xi) = \omega (x,\xi )$. Note also that $\omega (\eta) = 0$ for any
vector $\eta \in TTM^0$ that is tangent to some level set  $F(x,y) = \mathrm{const}.$ These two conditions characterize
the Hilbert form which is therefore independent of the choice of coordinates, see also \cite{Crampon2013}.
Observe also that  the length of a smooth non-singular curve $\beta : [a,b] \to M$ is 
$$
  \ell (\beta) = \int_a^b F(\tilde{\beta})= \int_{\tilde{\beta}} \omega,
$$
where $\tilde{\beta} : [a,b] \to TM^0$ is the natural lift of $\beta$.
We then have the

\begin{proposition}[\cite{Crampin2011}]
 The  smooth Finsler metric $F : T\mathcal{U} \to \r$ on the convex domain $\mathcal{U}$ is projectively flat
 if and only if the Hilbert form is $d_x$-closed, that is, if we have
 $$
   d_x \omega = \frac{\partial^2 F}{\partial x^i\partial y^j} dx^i\wedge dx^j = 0.
 $$
Equivalently, $F$ is projectively flat if and only if the Hilbert form is $d_x$-exact. This means that there exists
a function $h : T\mathcal{U}^0 \to \r$ such that 
$$
  \omega = d_xh = \frac{\partial h}{\partial x^j}dx^j.
$$
\end{proposition}

\begin{proof}
 The first assertion  is a mere reformulation of the Hamel Condition (c) in the previous Proposition. 
 The second assertion is proved using the same argument as in the proof of the Poincar Lemma.
 Indeed, using condition (c) from Proposition \ref{prop.hamel}, we compute:
 \begin{eqnarray*}
\frac{d}{dt}   \left(t \cdot  \omega_{(tx,y)}  \right) & = & \frac{d}{dt} \left( t\frac{\partial F}{\partial y^i}(tx,y)dx^i\right)
 \\ & = &  \frac{\partial F}{\partial y^i}(tx,y)dx^i  + t  \frac{\partial^2 F}{\partial x^k\partial y^i}x^kdx^i
 \\ & = &  \frac{\partial F}{\partial y^i}(tx,y)dx^i  + t \frac{\partial^2 F}{\partial x^i\partial y^k}x^kdx^i 
  \\ & = &  d_x \left( x^k \frac{\partial F}{\partial y^k}(tx,y)\right)
 \\ & = &  d_x   \left(\omega_{(tx,y)}(x)  \right) .
 \end{eqnarray*}
Suppose now that  $0 \in \mathcal{U}$ and set 
 $$
  h(x,y) = \int_0^1  \omega_{(tx,y)}(x) dt 
  = \int_0^1 \left(x^k \frac{\partial F}{\partial y^k}(tx,y)\right) dt ,
 $$
 then by the previous calculation we have
$$
 d_x h =   \int_0^1  d_x   \left(\omega_{(tx,y)}\right) dt =  \int_0^1  \frac{d}{dt}   \left(t \cdot  \omega_{(tx,y)}\right)  dt 
 =   \omega_{(x,y)}.
$$
\end{proof}

We shall call such a function $h$ a \emph{Hamel potential}\index{Hamel potential} of the projective Finsler metric $F$. 
Observe that a Hamel potential is well defined up to adding a function of $y$. The above proof shows that 
one can chose $h(x,y)$ to be $0$-homogenous in $y$.
This potential allows us to  compute distances.

\begin{corollary}
 The  distance $d_F$ associated to the projective Finsler metric $F$ is given by
 $$
   d_F(p,q) = h(q,q-p) - h(p,q-p),
 $$
 where $h(x,y)$ is a Hamel potential for $F$.
\end{corollary}

\begin{proof}
 Let $\beta (t) = p + t(q-p)$, ($t \in [0,1]$). Then 
 $$
    d_F(p,q) =\int_0^1 F(\beta, \dot \beta) dt = \int_{\tilde{\beta}} \omega = h(q,q-p) - h(p,q-p).
 $$
\end{proof}

\begin{example}
 If $F_f$ is the tautological Finsler structure in $\mathcal{U}$, then from the previous corollary and Equation (\ref{eq.distFFqp}) we deduce that 
 a Hamel potential is given by
 $$
  h(x,y) = - \log(F_f(x,y)).
 $$
 If one prefers a $0$-homogenous potential, then a suitable choice is 
 $$
  h(x,y) =  \log\left(\frac{F_f(p,y)}{F_f(x,y)}\right),
 $$ 
 where $p \in \mathcal{U}$  is some fixed point. 
\end{example}

\medskip

\begin{remark} A consequence of this example is that  the tautological Finsler structure in a domain $\mathcal{U}$
satisfies the following equation:
\begin{equation}\label{eq.FFxy}
 \frac{\partial F_f }{\partial x^j} = F_f \frac{\partial F_f}{\partial y^j} . 
\end{equation}
Indeed, since $h(x,y) =  - \log(F_f(x,y))$ is a Hamel potential, the Hilbert form is  
$$
    \frac{\partial F_f}{\partial y^j} dx^j =\omega = d_x h = \frac{1}{F_f}   \frac{\partial F_f}{\partial x^j} dx^j,
$$
from which (\ref{eq.FFxy}) follows imediately.  This  equation plays an important role in the study of
projectively flat metrics with constant curvature, see e.g. \cite{Berwald1929b}.
\end{remark}

\smallskip

An intrinsic discussion of  geodesics based on  the Hilbert form is given in 
\cite{Alvarez, Crampin2011,Crampin2013,Crampon2013}.

%________________
\section{Curvature in Finsler geometry}

A notion of curvature for Finsler surfaces already appeared in the beginning of the last century.
This notion was extended in all dimensions by L. Berwald in 1926  \cite{Berwald1926}.
This curvature is an  analogue of the 
sectional curvature in Riemannian geometry and it is best explained using the notion
of \emph{osculating Riemannian metric} introduced by Varga \cite{Varga1949}. See also
Auslander \cite{Auslander}  and  the book of Rund \cite[page 84]{Rund59}.

\medskip

\begin{definition}
(A) Let $(M,F)$ be a smooth manifold with a strongly convex Finsler metric and let  $(x_0,y_0) \in TM^0$. A vector field $V$
defined in  some neighborhood $\mathcal{O}\subset M$ of the point $x_0$ is said to be a \emph{geodesic extension}
of $y_0$ if $V_{x_0}= y_0$ and if  the integral curves of $V$ are geodesics of the Finsler metric $F$ (in particular
$V$ does not vanish throughout the neighborhood $\mathcal{O}$).

\smallskip

(B)  The \emph{osculating Riemannian metric}  $\g_V$ of $F$ in the direction of $V$ is the Riemannian metric on 
 $\mathcal{O}$ defined by the 
 Fundamental tensor at the point $(x,y) = (x,V_x) \in T\mathcal{O}^0$. In local coordinates we have
 $$
  \g_V = g_{ij}(x)dx^idx^j  =  g_{ij}(x,V(x))dx^idx^j   = \frac{1}{2} \frac{\partial^2F(x,V_x)}{\partial y^i\partial y^j}  \, dx^i dx^j.
 $$
\end{definition}

\medskip

Let us fix a point $(x_0,y_0) \in TM^0$ and a geodesic extension of $y_0$. We shall denote by $\Riem_V$ the $(1,3)$
Riemann curvature tensor of the osculating metric  $\g_V$. Recall that
$$
 \Riem_V(X,Y) Z = (\nabla_X\nabla_Y - \nabla_Y\nabla_X -\nabla_{[X,Y]}) Z,
$$
where $\nabla$ is the Levi-Civita connection of $\g_V$.

\medskip

\begin{definition}
 The \emph{Riemann curvature} of $\g_V$ is the field of endomorphisms ((1,1)-tensor) $\RR_V : TM \to TM$ defined as
 $$
   \RR_V (W) =  \Riem_V(W,V) V.
 $$
\end{definition}

A basic fact of Finsler geometry is the following result:
    
\begin{proposition}
  Let $(M,F)$ be a smooth manifold with a strongly convex Finsler metric and let $(x,y)$
  be a point in  $TM^0$. Then   the Riemann curvature $\RR_V$ at 
  $(x,y) \in TM^0$ is well defined independently of the choice of a geodesic extension $V$ of $y$.
\end{proposition}   
    
\begin{proof} Choose some local coordinate system and let us write $R^i_{\,\,\,k}(x,y)$ for the components
of the tensor $\RR_V$: 
$$
  \RR_V = R^i_{\,\,\,k}(x,y)  \, dx^k \otimes \frac{\partial}{\partial x^i}.
$$
Then we have the formula 
\begin{equation}\label{eq.Rik}
 R^i_{\,\,\,k}  =
 2\frac{\partial G^i}{\partial x^k}-\frac{\partial^2 G^i}{\partial x^j\partial y^k}y^j
 +2G^j \frac{\partial^2G^i}{\partial  y^j\partial y^k}
 -\frac{\partial G^i}{\partial y^j}\frac{\partial  G^j}{\partial y^k},
\end{equation}
where the $G^i= G^i(x,y)$ are the spray coefficients of $F$. 
We refer to Lemma 6.1.1 and Proposition 6.2.2 in \cite{Shen2001b} for a proof. Formula
(\ref{eq.Rik}) is also obtained in  \cite[page 43]{ChernShen}, where it is seen as
a consequence of the structure equations for the Chern connection,
see also  \cite[Proposition 8.4.3]{Shen2001a}.
Since the spray coefficients $G^i(x,y)$ only 
depend  on the fundamental tensor and its partial derivatives, it follows that
$R^i_{\,\,\,k} $ depends only on  $(x,y) = (x,V_x) \in TM^0$ and not on the
choice of a geodesic field extending $y$.
\end{proof}

This Proposition implies that  we can write 
 $$
   \RR_y  = \RR_V = R^i_{\,\,\,k}(x,y)  \, dx^k \otimes \frac{\partial}{\partial x^i}
 $$   
for the Riemann curvature at a point $(x,y) \in TM^0$, where $V$ is an arbitrary geodesic extension of $y$.
We then have the following important 
\begin{corollary}
 Let $\sigma \subset T_{x}M$ be a 2-plane containing the non-zero vector $y \in T_{x}M$. 
 Choose a local geodesic extension $V$  of  $y$, then the sectional curvature    
 $\K_{\g_V}(\sigma)$  of $\sigma$  for the osculating Riemannian metric $\g_V$
 is independent of the choice of $V$.
\end{corollary}

\begin{proof}
By definition of the sectional curvature in Riemannian geometry, we have
 \begin{equation} \label{Flag1}
  \K_{\g_V}(\sigma) = \frac{\g_V(\RR_V(W), W)}{\g_V(y,y)\g_V(W,W)-\g_V(y,W)^2},
\end{equation}
where $W\in \sigma$ and  $y = V_x$ are linearly independent (and thus $\sigma = \mbox{span} \{V,W\}$), and 
 $\RR_V$ is the  Riemann curvature    of $\g_V$.  
This quantity is independent of the geodesic extension $V$, by the previous Proposition.
\end{proof}
  
 \medskip
 
\begin{definition} 
 The pair $(y,\sigma)$ with $0 \neq y \in \sigma$  is called a \emph{flag} in $M$, and $\K_{\g_V}(\sigma)$ 
 is called the \emph{flag curvature}\index{flag curvature} of $(y,\sigma)$, and denoted by $\K(y,\sigma)$. 
 The vector $y\in \sigma$ is sometimes suggestively called the \emph{flagpole}  of the flag.
\end{definition}
 
In local coordinates, the flag curvature can be written as
 \begin{equation} \label{Flag2}
   \K(y,\sigma)  
  =   \frac{R_{mk}(x,y) \, w^kw^m}{F^2(x,y)g_{ij}(x,y)w^i w^j - (g_{rs}(x,y)w^ry^s)^2},
\end{equation}
where  $W = w^k\frac{\partial}{\partial x^k} \in \sigma$ and  $y$ are linearly independent and  
$R_{mk} = g_{im}R^i_{\,\,\,k}$.

\medskip 

\begin{example}
The flag curvature of a Minkowski metric is zero. Indeed, the fundamental tensor is constant and 
coincides with any osculating metric which is thus flat.
Note that conversely, there are many examples of  Finsler metrics with vanishing
flag curvature which are not locally isometric to a Minkowski metric.
The first example has been given in \cite[section 7]{Berwald1929}.
\end{example}

We also have a notion of Ricci curvature:

\begin{definition}
 The \emph{Ricci curvature} of the Finsler metric $F$ at $(x,y) \in TM^0$ is defined as
 $$
  \mathrm{Ric}(x,y) =  \mbox{Trace}(\RR_y) =F^2(x,y)\cdot \sum_{i=2}^n\K(y,\sigma_i) 
 $$
 where $e_1, e_2, \dots, e_n \in T_xM$  is an orthonormal basis relative to the
 inner product $\g_y$ \, such that $e_1=\frac{y}{F(x,y)}$.
\end{definition}

\medskip

The geometric meaning of the Finslerian flag curvature presents both similarities and striking differences
with its Riemannian counterpart. 
The Riemann curvature  $\RR_y$ plays an important role in the Finsler literature, it appears naturally in the second variation
formula for the length of geodesics and in the theory of Jacobi fields (see \cite[Lemma 6.1.1]{Shen2001b} 
or \cite[chap. IV.4 and IV.5]{Rund59}). This leads to natural formulations of comparison theorems in
Finsler geometry  that are similar to their Riemannian counterparts. In particular we have

\medskip

\begin{itemize}
  \item In 1952, L. Auslander proved a Finsler version of the Cartan-Hadamard Theorem. If the Finsler Manifold $(M,F)$ is forward complete and has non-positive
  Flag curvature, then the exponential map is a covering map \cite{Auslander, Moalla}.
  \item Auslander also proved a Bonnet-Myers Theorem: A forward complete Finsler manifold $(M,F)$ with Ricci curvature $\mbox{Ric}(x,y) \geq (n-1)F^2(x,y)$  for all $(x,y) \in TM^0$ is compact with diameter $\leq \pi$, see  \cite{Auslander, Moalla}.  
  \item In 2004, H.-B. Rademacher proved a sphere theorem: 
   a compact,  simply-connected Finsler manifold of dimension $n \geq 3$ such that $F(p,-\xi) \leq \lambda F(p,\xi)$ 
  for any $(p,\xi) \in TM$ and   with flag curvature satisfying 
  $$
   \left(1-\frac{1}{1+\lambda}\right)^2 < \K \leq 1
  $$
is homotopy equivalent (and thus homeomorphic) to a sphere  \cite{Rademacher2004} (see also  \cite{Dazord} for an earlier result in this direction).
Note that if $F$ is reversible, i.e. $\lambda = 1$, then we have the analog of the familiar  $\frac{1}{4}-$pinching sphere theorem of Riemannian geometry.
\end{itemize}
 
We also mention the  Finslerian version of the Schur Lemma: 

\begin{lemma}[The Schur Lemma]\index{Shur lemma}\index{lemma!Schur}
 Let $(M,F)$ be a smooth manifold of dimension $n \geq 3$ with a strongly convex Finsler metric. 
 Suppose that at each point $p\in M$ the flag curvature is independent of the flag, that is, 
 $$
   \K(y,\sigma) = \kappa (p),
 $$
 for every flag $(y,\sigma)$ at $p$, where $\kappa : M \to \r$ is an arbitrary function. Then 
 $\K$ is  a constant,
 that is $\kappa$ is  independent of $p$.
 \end{lemma}

\medskip

We refer to \cite[Lemma 3.10.2]{BCS} for a proof. This result is already stated without proof in 
the work of Berwald (see the footnote on page 468 in  \cite{Berwald1929}).

\medskip

Finally we should warn the reader that  unlike the situation in Riemannian geometry, the flag curvature does not control the purely metric notions of curvature such as the notions  of nonpositive (or non positive) curvature in the sense of  Busemann or Alexandrov.
In particular we shall prove below that Hilbert geometries
have constant negative flag curvature, yet they satisfy the Busemann or Alexandrov curvature condition if and only if the convex 
domain is an ellipsoid, see \cite{RGuo2013}.

\section{The  curvature of  projectively flat Finsler metrics} \label{sec.curvPflat}

The Riemann curvature of a general projectively flat Finsler metric was computed by Berwald in \cite{Berwald1929}.
We can state the result in the following form:

\medskip 

\begin{theorem}
 Let $F(x)$ be a strongly convex projectively flat Finsler metric on some domain $\mathcal{U}$ of $\r^n$. Then the Riemann curvature at
 a point $(x,y) \in T\mathcal{U}^0$ is given by
\begin{equation}\label{eq.Ryplat}
 \RR_{y} = \Sc (x,y)  \PP_{y^{\perp}},
\end{equation}
where   
\begin{equation}\label{eq.Rscal}
  \Sc (x,y) =  \left(P^2 - y^j\tfrac{\partial P}{\partial x^j}  \right)
\end{equation}
and $\PP_{y^{\perp}} : T_x\mathcal{U} \to T_x\mathcal{U}$ is  the orthogonal projection onto the hyperplane $y^{\perp} \subset T_x\mathcal{U}$
 with respect to the fundamental tensor $\g_y$ at $(x,y)$.
\end{theorem}

\begin{proof}
We basically follow  the proof in  \cite[page 110]{ChernShen}. 
Since $F$ is projectively flat, we have $G^i = P(x,y)y^i$ and equation (\ref{eq.Rik})   gives
$$
   \RR_y   = R^i_{\,\,\,k}(x,y)  \, dx^k \otimes \frac{\partial}{\partial x^i}.
$$
with
\begin{equation*}
 R^i_{\,\,\,k}  =
 2\frac{\partial (Py^i)}{\partial x^k}-\frac{\partial^2 (Py^i)}{\partial x^j\partial y^k}y^j
 +2(Py^j) \frac{\partial^2(Py^i)}{\partial  y^j\partial y^k}
 -\frac{\partial (Py^i)}{\partial y^j}\frac{\partial  (Py^j)}{\partial y^k}.
\end{equation*}
Using the homogeneity in $y$ for $P(x,y)$ and $\frac{\partial y^i}{\partial y^k} = \delta^i_k$, we calculate
that
$$
 R^i_{\,\,\,k}  =  \Sc \, \delta^i_k  + \Tau_k y^i,
$$
where
$
  \Sc  =  \left(P^2 - y^j\frac{\partial P}{\partial x^j}  \right),
$
and
\begin{eqnarray*}
  \Tau_k  
     & = &
  2P\frac{\partial P}{\partial y^k}   - \frac{\partial P}{\partial y^k}  - \frac{\partial^2 P}{\partial y^k\partial x^j}y^iy^k
 +
  3\left(\frac{\partial P}{\partial x^k}  - P  \frac{\partial P}{\partial y^k} \right)
  \\
  & = & \frac{\partial \Sc}{\partial y^k} + 
  3\left(\frac{\partial P}{\partial x^k}  - P  \frac{\partial P}{\partial y^k} \right).
\end{eqnarray*}

 Observe that
\begin{equation}\label{eq.SL1}
  y^k \Tau_k  = - \Sc,
\end{equation}
therefore
$
 R^i_{\,\,\,k}  =  \Sc\delta^i_k  + \Tau_k y^i,
$
which we write as
$$
 R_{mk} = g_{mi} R^i_{\,\,\,k}  =  g_{mi} ( \Sc\delta^i_k  + \Tau_k y^i)
 =  \Sc g_{mk}  + g_{mi}y^i \Tau_k .
$$
We shall compute $\Tau_m$ using a trick. From the symmetry 
$R_{mk} = R_{km}$ we find
$$
 0 = (R_{mk} - R_{km}) = 
   (g_{ki} - g_{mi})y^i \Tau_m,
$$  
therefore
$
   g_{mi}y^i \Tau_k  = g_{ki}y^i \Tau_m.
$   
Using (\ref{eq.SL1}), one gets
\begin{equation}\label{eq.SL2}
  g_{mi}y^i  \Sc  =
    - g_{mi}y^i y^k\Tau_k  = - g_{ki}y^iy^k \Tau_m =  - F^2 \Tau_m,
\end{equation}
that is
$$
 \Tau_k = - \Sc (x,y) \frac{g_{kj}(x,y)}{F^2(x,y)} \, y^j.
$$
Let  $\PP_y : T_x\mathcal{U} \to T_x\mathcal{U}$ denotes the orthogonal projection 
on the line $\r y \subset T_xM$, then
we have
$$
 \PP_y\left(\frac{\partial}{\partial x^k}\right) = \frac{\g_y(\partial_k,y)}{\g_y(y,y)}\cdot y =   
 \frac{g_{kj}(x,y) \, y^j}{F^2(x,y)}\cdot y
  =\frac{ \Tau_ky^i}{\Sc}\,   \frac{\partial}{\partial x^i},
$$ 
and finally
$$
 \RR_y = \Sc (x,y) \cdot \left(\mathrm{Id} -\PP_y\right) = \Sc (x,y) \cdot \PP_{y^{\perp}}.
$$
\end{proof}

\begin{remark}
The  coefficient $\Tau_k$ is also expressed as 
$
 \Tau_k = - \frac{\Sc}{F} \frac{\partial F}{\partial y^k}.
$
This is equivalent to (\ref{eq.SL2})  since we have 
 by homogeneity
$
 F\frac{\partial F}{\partial y^k} = g_{ki}y^i.
$ 
\end{remark}

\begin{corollary}[Berwald \cite{Berwald1929}] \label{cor.berwald1}
 The flag curvature of a projectively  flat strongly convex Finsler metric $F$ at a point $(x,y) \in TM^0$ is given by
 \begin{equation}\label{eq.Kprojflat}
   \K(y,\sigma) = \frac{1}{F^2}\left(P^2 - y^j\frac{\partial P}{\partial x^j}  \right),
\end{equation}
 where $P = P(x,y)$ is the projective factor and $\sigma \subset TM$ is an arbitrary $2$-plane containing $y$. 
 The Ricci curvature of a projectively Finsler metric $F$   is given by
 \begin{equation}\label{eq.Ricprojflat}
 \mathrm{Ric}(x,y) =  (n-1) \, \left(P^2 - y^j\frac{\partial P}{\partial x^j}  \right).
\end{equation}
\end{corollary}

\begin{proof}
From the previous proposition  we have for any $\mathrm{w} \in T_x\mathcal{U}$:
\begin{equation*}\label{eq.RScal}
 \RR_{y}(\mathrm{w}) =  \Sc (x,y) \left(\mathrm{w} - \frac{\g_y(\mathrm{w},y)}{\g_y(y,y)}y  \right)
 = \frac{\Sc (x,y)}{\g_y(y,y)} \left(\g_y(y,y) \mathrm{w} -  \g_y(\mathrm{w},y) y  \right)
\end{equation*}
 where  $\Sc (x,y) =  \left(P^2 - y^j\frac{\partial P}{\partial x^j}  \right)$.
Let $\sigma\subset T_x\mathcal{U}$ be a $2$-plane containing $y$ and choose a vector
 $\mathrm{w} \in T_x\mathcal{U}$ such that 
$\sigma = \text{span}(y,\mathrm{w})$, then
\begin{align*}
\K(y,\sigma)  &= 
\frac{\g_y(\RR_y(\mathrm{w}),y)}{\g_y(y,y)\g_y(\mathrm{w},\mathrm{w})- \g_y(\mathrm{w},y)^2}  
=  \frac{\Sc (x,y)}{\g_y(y,y)} =  \frac{\Sc (x,y)}{F^2(x,y)}.
\end{align*}
For the Ricci curvature we have
 $$
  \mathrm{Ric}(x,y) =   \mbox{Trace}(\RR_y) =  \Sc (x,y) \,  \mbox{Trace}(\PP_{y^{\perp}})
  =  (n-1)\Sc (x,y).
 $$
\end{proof}

\medskip

\begin{remark}\label{Kscal}
Observe in particular that the flag curvature of a
projectively flat Finsler metric at a point $(x,y) \in TM^0$ depends only on 
$(x,y)$ and not on the $2$-plane $\sigma  \in TM^0$. Such metrics are 
said to be \emph{of scalar flag curvature}, because the flag curvature
is given by a scalar function
$
 \K : TM^0 \to \r.
$
In that case we write the flag curvature as 
$$
  \K(x,y) = \K(y, \sigma).
$$
\end{remark}

If the flag curvature is independent of $y$, then it is in fact constant:
 
\begin{proposition}
 Let $(M,F)$ be a smooth manifold with a strongly convex Finsler metric.  
 If the flag curvature  $\K(x,y)$ is independent of $y\in T_xM$ for any $x\in M$, 
 then the flag curvature is actually a constant.
\end{proposition}

We omit the proof. In dimension $\geq 3$ this is a particular case of the Schur Lemma. 
Berwald gave an independent proof
in all dimensions in \cite[\S 9]{Berwald1929}.

\medskip

In the same spirit, we show how Corollary \ref{cor.berwald1} can be used to prove the Beltrami Theorem of 
Riemannian geometry in dimension at least $3$.

\begin{theorem}[Beltrami]  \index{Beltrami theorem}\index{theorem!Beltrami} \label{th.Beltrami}
 A connected Riemannian manifold $(M,g)$ is locally projectively flat if and only if it has
 constant sectional curvature.
\end{theorem}

\begin{proof}
 If $\dim(M) = 2$, then a direct proof of the fact that its curvature is constant is not difficult, see e.g. Busemann   \cite[page 85]{Busemann1955}.
 We can thus assume $\dim(M) \geq 2$.
 Fix a point $p$ and consider two $2$-planes $\sigma, \sigma' \subset T_pM$ and choose non-
 zero vectors $y\in \sigma$ and $y' \in \sigma'$. Let us also set  $\tau = \mathrm{span} \{y,y' \} \subset T_pM$. 
 Because $(M,g)$ is Riemannian, the flag curvature $\K_p(y,\sigma)$ of a the flag $(y,\sigma)$ is 
 independent of the choice of the flagpole $y \in \sigma$, and because $M$ is projectively flat, 
 Corollary \ref{cor.berwald1} says that the flag curvature of a  flag $(y,\sigma)$ is independent the 
 $2$-plane $\sigma$ containing $y$.
 Therefore
 $$
   \K_p(\sigma) = \K_p(y,\sigma) = \K(p,y) =  \K_p(\tau,y) = \K(p,y') = \K_p(\sigma',y') 
   = \K_p(\sigma').
 $$
 The sectional curvature of $(M,g)$ at a point $p$ is thus independent of the choice of
 a $2$-plane $\sigma \subset T_p.$ Because $M$ is connected we then conclude from the
 Schur Lemma that $(M,g)$ has constant sectional curvature.
\end{proof}

\begin{remark}
For a modern Riemannian proof, see \cite{Matveev2006}, or \cite[chapter 8, Theorem 4.2]{Sharpe}
for a proof from the point of view of Cartan geometry. 
The converse of this theorem is classical. Suppose $(M,g)$ has constant sectional curvature $K$. 
 Rescaling the metric if necessary, we may assume $K = 0$, $+1$ or $-1$.  If $K = 0$,
then $(M,g)$  is locally isometric to the Euclidean space $\r^n$, which is flat. If $K=+1$, then $(M,g)$
 is locally isometric  to the standard sphere $S^n$, which is projectively flat by Example (\ref{ex.Fplat1}.c). 
 And if $K=-1$, then $(M,g)$ is locally isometric to the hyperbolic space $\mathbb{H}^n$,
 which is isometric to the unit ball $\mathbb{B}^n$ with its Klein metric  (\ref{KleinMetric}).
\end{remark}

\section{The flag curvature of the Funk and the Hilbert geometries}

The flag curvature of the Hilbert metrics was computed in 1929 by Funk in dimension 2  and  by Berwald
in all dimensions, see \cite{Berwald1929,Funk1929}. In 1983, T. Okada proposed a more direct computation \cite{Okada},
and in 1995,   D. Egloff  related these curvatures to the Reeb field of the Finsler  manifold \cite{EgloffPhD, Egloff1997}.

\medskip

The original Funk-Berwald computation is based on the following important relation between the flag curvature and the 
exponential map of a projectively flat Finsler manifold:

\begin{proposition}\label{Prop.Funk}
 Let $F : T\mathcal{U} \to \r$ be a strongly convex projectively flat  Finsler metric on 
 the convex domain $\mathcal{U}$ of $\r^n$. 
 Then for any geodesic $\beta(s) = p + \varphi(s) \xi$ we have
 $$
  \dot \varphi (s)^2 \cdot  \Sc (\beta (s), \xi)  =  \tfrac{1}{2}  \left. \{ \varphi (s), s\} \right|_{s=0}
 $$
 where $ \{ \varphi (s), s\}$ is the Schwarzian derivative\footnote{See Appendix B.} of $\varphi$ with respect to $s$.
\end{proposition}

\medskip

\begin{proof}
 Let us write  $x(s) = \beta (s) =  p + \varphi(s) \xi$ and $y(s) = \dot x (s) = \dot \varphi (s) \xi$. 
 The geodesic equation  (\ref{eq.geod2}) implies 
 $$
   \dot y^k  + 2 G^k (x,y) =  \dot y^k  + 2 P^k (x,y) y^k = \left(\ddot \varphi(s) + 2 \dot\varphi (s) P(x,y) \right) \xi^k = 0,
 $$
 for $k = 1, 2, \dots, n$. We thus have
 $$
    \frac{\ddot \varphi (s) }{\dot \varphi(s)} =  -2 P(x,y),
 $$
 and therefore
 $$
  \frac{d}{ds} \left( \frac{\ddot \varphi }{\dot \varphi }\right) = -2  \frac{d}{ds} P(x,y) 
  = - 2 \frac{\partial P}{\partial x^k} y^k - 2\frac{\partial P}{\partial y^k} \dot y^k. 
 $$
 Since $y^k =\dot x^k =  \dot \varphi(s) \xi^k$ and $\dot y^k = \ddot \varphi(s) \xi^k$, we have  
 $$
  \dot y^k =  \frac{\ddot \varphi }{\dot \varphi} \cdot y^k = -2 P(x,y)\,  y^k
 $$ 
  and 
 $$
   \frac{d}{ds} \left( \frac{\ddot \varphi }{\dot \varphi}\right) =
   - 2 \frac{\partial P}{\partial x^k} \dot x^k + 4 P\, \frac{\partial P}{\partial y^k}  y^k =  - 2 \frac{\partial P}{\partial x^k} y^k + 4 P^2,
 $$
 because $P(x,y)$ is homogenous of degree 1 in $y$. We thus obtain
 $$
   \big\{ \varphi , s\big\} = \frac{d}{ds} \left( \frac{\ddot \varphi }{\dot \varphi }\right) - \frac{1}{2} \left( \frac{\ddot \varphi }{\dot \varphi }\right)^2
   = 2\left(P^2 - \frac{\partial P}{\partial x^k} y^k \right) = 2\Sc (x,y) = 2 \dot\varphi(s)^2 \Sc (x, \xi).
 $$
\end{proof}

\bigskip

A first consequence of  this proposition is the following local characterization of reversible Minkowski metrics:

\begin{corollary}\label{cor.Knul}
 Let $F(x,y)$ be a strongly convex reversible Finsler metric. Then $F$ is locally Minkowski
 if and only if it is projectively flat with flag curvature $\K = 0$.
\end{corollary}
 
\begin{proof}  
The geodesics of a  (strongly convex)  Minkowski metric $F$ are the affinely parametrized straight lines
$\beta (s) = p + (as+b)\xi$. Therefore $F$ is projectively flat and its flag curvature vanishes since
$$
 \{ as+b, s\} = 0.
$$
Conversely, suppose that $F(x,y)$ is a strongly convex projectively flat reversible Finsler metric with flag curvature $\K = 0$.
The geodesics are then of the type $\beta (s) = p + \varphi(s)\xi$, with $ \{ \varphi(s), s\} = 0$. 
Using Lemma \ref{lem.schwarzian} one obtains 
$$
  \varphi(s) = \frac{As+B}{Cs+D}, \qquad (AD-BC \neq 0).
$$
Assuming the initial conditions $\beta (0) = p$ and $\dot \beta(0) = \xi$, we get $B= 0$ and $A=D$. Since 
$F$ is reversible,  $\beta^{-} (s) = p - \varphi(s)\xi$ is also a geodesic and it has the same 
speed as $\beta$, therefore $\varphi(-s) = - \varphi(s)$
and thus
$$
 \varphi(s) =  \frac{As}{Cs+A} = - \varphi(-s)   =  \frac{As}{-Cs+A}.
$$
This implies $C = 0$, therefore $\varphi(s) = s$ and the projective factor $P(p,\xi) = 0$.
We conclude form Corollary  \ref{cor.PnulMink} that $F$ is locally Minkowski.
\end{proof}

This result  is due to Berwald,  who gave a different proof.  Berwald  also gave a counterexample in the non-reversible
case, see  \cite[\S 7]{Berwald1929} and \cite{Shen2001d}. 
 
 \medskip
 
 Another consequence of the Proposition is the following calculation of the flag curvature of the Funk and Hilbert geometries:
 
\begin{corollary}
 The flag curvature of the Funk metric
 in a strongly convex domain $\mathcal{U}\subset \r^n$ is constant equal to $-\frac{1}{4}$
 and the  flag curvature of the Hilbert metric in $\mathcal{U}$ is constant equal to $- 1$.
\end{corollary} 

\begin{proof}
These  geometries are projectively flat with geodesics $\beta (s) = p + \varphi (s) \xi$ and from the previous proposition, we 
know that
$$
   \K (p, \xi) = \frac{\Sc (p, \xi)}{F^2(p, \xi)}
   =   \frac{1}{2}   \frac{\left.\big\{ \varphi (s), s\big\} \right|_{s=0}}{\dot \varphi (0)^2 F^2(p, \xi)}.
 $$
For the Funk metric, the function  $\varphi (s)$ is given by (\ref{eq.FunkGeodesic1}) and we have
$\big\{ \varphi , s\big\} = -\frac{1}{2}$ and $\dot \varphi (0) F(p, \xi) = 1$, therefore
$\K = -\frac{1}{4}$.
For the Hilbert metric,
 the function  $\varphi (s)$ is given by (\ref{eq.HilbGeodesic1a}). We calculate that
$\big\{ \varphi(s) , s\big\} = - 2$ and $\dot \varphi (0) F(p, \xi) = 1$, and we thus obtain $\K = -1$.
\end{proof}

 \bigskip   

\begin{remark}\label{rem-Okadacalc} Following  Okada \cite{Okada}, we can also compute these curvatures directly from Corollary
\ref{cor.berwald1} and Equation  (\ref{eq.FFxy}). 
 For the Funk metric, we have $P(x,y) = \frac{1}{2}F_f(x,y)$, therefore equation (\ref{eq.FFxy}),
 gives
 $$
    \frac{\partial P}{\partial x^j}y^j  = 2 P \frac{\partial P}{\partial y^j}y^j =  2P^2,
 $$
 and we thus have
 $ \K = \frac{1}{4 P^2}\left(P^2 -  \frac{\partial P}{\partial x^j}y^j \right)
  = -  \frac{1}{4}.$
 For the Hilbert metric $F_h = \frac{1}{2}(F_f + F_f^*)$, the projective factor is 
 $
   P(x,y) =  \frac{1}{2}(F_f - F_f^*),
 $
 therefore
 $$
    \frac{\partial P}{\partial x^j} =  \frac{1}{2}(\frac{\partial F_f}{\partial x^j} - \frac{\partial F_f^*}{\partial x^j})
    =  \frac{1}{2}(F_f\frac{\partial F_f}{\partial y^j} + F_f^*\frac{\partial F_f^*}{\partial y^j}),
 $$
 and thus
 $$
     \frac{\partial P}{\partial x^j}y^j = \frac{1}{2}F_f F_h =  \frac{1}{2}(F_f^2 + {F_f^*}^2).
$$
It follows that 
\begin{eqnarray*}
P^2 -  \frac{\partial P}{\partial x^j}y^j  & = &
\frac{1}{4}(F_f-F_f^*)^2 - \frac{1}{2}(F_f^2 + {F_f^*}^2)
\\ & = & 
 -\frac{1}{4}(F_f+F_f^*)^2
\\ & = &  -F_h^2.
\end{eqnarray*}
And we conclude that
$
  \K = \frac{1}{F_h^2}\left(P^2 -  \frac{\partial P}{\partial x^j}y^j \right)
  = -  1.
$
\end{remark}

This curvature computation allows us to provide a simple proof of the following important result
which is due to I. J. Schoenberg and D. Kay, see \cite[Corollary 4.6]{Guo2012}. 

\begin{theorem}
 The Hilbert metric $F_h$ in a bounded  domain $\mathcal{U} \subset \r^n$ with smooth strongly convex boundary
 is Riemannian  if and only if \  $\mathcal{U}$ is an ellipsoid. 
\end{theorem}

\begin{proof}
 If  $\mathcal{U}$ is an ellipsoid, then it is affinely  equivalent to the unit ball $\mathbb{B}^n$, therefore $F_h$
 is equivalent to the Klein metric (\ref{ex.Klein}), which is Riemannian. 
 Suppose conversely that $F_h$ is Riemannian, that is $F_h = \sqrt{\g}$ for some Riemannian metric in $\mathcal{U}$.
 Then $(\mathcal{U},\g)$ is a complete, simply-connected Riemannian manifold with constant sectional curvature $\K=-1$,
 therefore $(\mathcal{U},\g)$ is isometric to the hyperbolic space $\mathbb{B}^n$ with its Klein metric and it follows that 
 $\mathcal{U}$ is an ellipsoid.
\end{proof}

\begin{corollary}
Let \  $\mathcal{U} \subset \r^n$ be a bounded smooth strongly convex domain.
Assume that there exists a discrete group $\Gamma$ of projective transformations leaving $\mathcal{U}$ invariant and acting 
freely with compact quotient $M= \mathcal{U}/\Gamma$ 
(the convex domain $\mathcal{U}$ is then said to be \emph{divisible}). Then  $\mathcal{U}$ is an ellipsoid.
\end{corollary}

\begin{proof}
 The Hilbert metric induces a smooth and strongly convex Finsler metric $F$ on the quotient $M= \mathcal{U}/\Gamma$.
 The compact Finsler manifold $(M,F)$ has constant  negative flag curvature, it is therefore Riemannian  by Theorem \ref{th.Akbar}.
 It now follows from the previous theorem  that the universal cover $\mathcal{U} = \widetilde{M}$  is an ellipsoid.
 \end{proof}

This corollary is a special case of a  result of Benz\'ecri \cite{Benzecri}. 
There are several generalizations of this result and divisible convex  domains have been  
the subject of intensive research in recent years, see \cite[Section 7]{RGuo2013} 
 and \cite{Marquis2013} for a discussion.

\section{The Funk-Berwald characterization of Hilbert geometries} \label{sec.charH}

We are now in a position to prove our characterization theorem for Hilbert metrics. 

\begin{theorem}\label{th.CharHilb}
Let $F$ be a strongly convex Finsler metric on a bounded convex domain  $\mathcal{U} \subset \r^n$.
Then $F$ is the Hilbert metric of that domain if and only if $F$ is projectively flat, complete 
and has constant flag curvature $\K = -1$.
\end{theorem}

\medskip

\begin{proof}
Since the Hilbert metric in a bounded  
convex domain $\mathcal{U} \subset \r^n$ is complete with flag curvature $\K = -1$, we can reformulate
the theorem as follows:
\emph{Let $F_1$ and $F_2$ be two  strongly convex  Finsler metrics with constant flag curvature $\K = -1$ on a bounded  
 convex domain $\mathcal{U} \subset \r^n$. Suppose that  $F_1$ and $F_2$ are complete and projectively flat, then $F_1 = F_2$.}

\medskip

To prove that statement, let us consider two 
complete and projectively flat strongly convex Finsler metrics  $F_1$ and $F_2$ in $\mathcal{U}$ such that
$$
 \K_{F_1}(x,y) =  \K_{F_2}(x,y) =-1
$$
for any $(x,y) \in T\mathcal{U}$. Fix a point  $p\in \mathcal{U}$ and a non-zero vector $\xi \in T_p\mathcal{U}$,
and let 
$$
  \beta_i (s) = p + \varphi_i(s) \xi
$$
be the unit speed geodesic for the metric $F_i$ starting at $p$ in the direction $\xi$ for $i=1,2$.
In particular we have $\varphi_1(0)=\varphi_2(0) = 0$ and $\dot \varphi_i(s) >0$.

Since $F_i$ is forward and backward complete,  we have
$$
 \beta_i(\pm \infty) = \lim_{s \to \pm\infty}  \beta_i(s) \in \partial \mathcal{U},
 \qquad (i=1,2).
$$  
By the definitions of the tautological and reverse
tautological Finsler structures, we obtain from Equation (\ref{eq.FunkGeodesic1})  that
\begin{equation}\label{eq.relFinf}
 \varphi_i(+\infty) = \frac{1}{F_f(p,\xi)} \quad \text{and} \quad
  \varphi_i(-\infty) = \frac{-1}{F^*_f(p,\xi)},
\end{equation}
for $i=1,2$. In particular $\varphi_1(+\infty) =\varphi_2(+\infty)$ and 
$\varphi_1(-\infty) =\varphi_2(-\infty)$.

Because $\beta_1$ and $\beta_2$ are unit speed geodesics, we have from Proposition (\ref{Prop.Funk}).
$$
 \frac{1}{2} \big\{ \varphi_i , s\big\} = F_i^2(\beta_i(s), \beta_i(s))\cdot \K_{F_i}(\beta_i(s), \dot\beta_i(s))
 =-1.
$$
In particular $\big\{ \varphi_1 , s\big\} = \big\{ \varphi _2, s\big\}$.
From  Lemma \ref{lem.schwarzian}, we have therefore
$$
   \varphi_2(s)  =  \frac{A\varphi_1(s) +B}{C\varphi_1(s) +D}.
$$
Since $\varphi_1(0)=\varphi_2(0) = 0$ we have $B=0$; and Equation  (\ref{eq.relFinf}) implies the
relations
$$
 A-D = 
 \frac{C}{F_f(p,\xi)} = \frac{-C}{F^*_f(p,\xi)}.
$$
It follows that $C=0$ and $A=D$ and therefore $\varphi_1(s) =\varphi_2(s)$
for every $s\in \r$, which finally implies that $F_1(p,\xi) = F_2(p,\xi)$
for all $(p,\xi) \in T\mathcal{U}$.
\end{proof}

 \appendix
\section{On Finsler metrics  with constant flag curvature}  \label{sec.Related}

The previous discussion suggests the following program: \emph{Describe all 
Finsler metrics with constant flag curvature. }

\medskip

This program is quite broad and it is not clear if  a complete answer will be at hand in the near future, 
but many examples and partial classifications have  been obtained. 
Let us only mention a few classic results. For the flat case, we have the following:

\begin{theorem}
 Let $(M,F)$ be a smooth manifold with a strongly convex Finsler metric 
 of constant flag curvature $\K = 0$.
 Suppose that either 
 \begin{enumerate}[i.)]
  \item $F$  is reversible and locally  projectively flat, or
  \item $F$ is Randers, or
  \item $F$ is Berwald, or
  \item $M$ is compact.
\end{enumerate}
Then $(M,F)$ is flat, that is it is locally isometric to a Minkowski space.
\end{theorem}

The first case is Corollary \ref{cor.Knul}. The second case is due to Shen
\cite{Shen2003b}, while the  case of a Berwald metric is classical (see Theorem 2.3.2 in \cite{ChernShen}).
Finally the compact case is due to  Akbar-Zadeh  \cite{Akbar}.

\medskip

Another notable result is
\begin{theorem}[Akbar-Zadeh \cite{Akbar}] \label{th.Akbar}
 On a smooth compact manifold, every strongly convex Finsler metric 
 with negative flag curvature is a Riemannian metric.
\end{theorem}
A proof is also given in \cite[page 162]{Shen2001a}. 
For positive constant curvature, we have the following recent result by  Kim and  Min  \cite{Kim2009}:
\begin{theorem}
 Any strongly convex  reversible Finsler  metric with positive constant flag curvature is Riemannian.
\end{theorem}
In the case of  projectively flat reversible Finsler structures on the  $2$-sphere, the result is due to 
R. Bryant \cite{Bryant1997}, who in fact classified all  projectively flat Finsler structures on the $2-$sphere
with flag curvature $\K = +1$.

\medskip

Non reversible projectively flat Randers metrics with constant curvature have been classified by  Shen in  \cite{Shen2003b}.
For further examples and discussions, we refer to \cite[Chapter 8]{ChernShen}, 
 \cite[\S 11.2]{Shen2001a}, \cite[Chapter 9]{Shen2001b}, \cite{Shen2003a,Mo2011}
and the recent survey  \cite{Guo2012} by Guo,  Mo and  Wang.

%_________
\section{On the Schwarzian derivative}
 
The Schwarzian derivative is a third order non-linear differential operator that is invariant under the group of one-dimensional
projective transformations. It appeared first in complex analysis  in the context of the  Schwarz-Christoffel transformations
formulae. It is defined as follows.

\begin{definition}
 Let $\varphi$ be a $C^3$ function of the real variable $s$, or a holomorphic function of the complex variable $s$. 
 Its \emph{Schwarzian derivative}\index{Schwarzian derivative} is defined as
 $$
   \big\{ \varphi (s), s\big\} = \frac{\dddot \varphi (s)}{\dot \varphi (s)} - \frac{3}{2} \left( \frac{\ddot   \varphi (s)}{\dot \varphi (s)}\right)^2
   = \frac{d}{ds} \left( \frac{\ddot \varphi (s)}{\dot \varphi (s)}\right) - \frac{1}{2} \left( \frac{\ddot \varphi (s)}{\dot \varphi (s)}\right)^2,
 $$
 where the dots represent differentiation with respect to $s$.
\end{definition}
 
For instance $\big\{e^{\lambda s}, s\big\} = - \frac{1}{2}\lambda^2$, \  $\big\{\frac{1}{s} ,s\big\} = 0$ and 
$\big\{\tan (\lambda s), s\big\} =2\lambda^2$. 
The fundamental property of the  Schwarzian derivative is contained in the following

\begin{lemma}\label{lem.schwarzian}
 (i) Let $u(s)$ and $v(s)$ be two linearly independent solutions to the equation
\begin{equation}\label{eq.w}
 \ddot w(s) +  \frac{1}{2}\rho(s) \, w(s) = 0.
\end{equation}
If $v\neq 0$, then $\ds \varphi(s) = u(s)/v(s)$ is a solution to
$$
    \big\{ \varphi (s), s\big\} = \rho (s).
$$
 (ii) Let $\varphi_1(s)$ and $\varphi_2(s)$ be two non-singular functions of $s$. Then 
 $$\big\{ \varphi_1(s) , s\big\} = \big\{ \varphi _2(s), s\big\}$$ 
 if and only if there exist constants $A,B,C,D \in \r$
 with $AD-BC \neq 0$ such that:
 $$
   \varphi_2  =  \frac{A\varphi_1 +B}{C\varphi_1 +D}.
 $$
\end{lemma}

\begin{proof}  The proof of (i) follows from a direct calculation. To prove (ii), we first observe that 
$\big\{-\varphi , s\big\} = \big\{ \varphi , s\big\}$ and we may therefore  assume $\dot \varphi_i > 0$. Let us set 
$$
  u_i(s) =   \frac{\varphi_i (s)}{\sqrt{\dot \varphi_i (s)}},
  \qquad 
  v_i(s) =   \frac{1}{\sqrt{\dot \varphi_i (s)}},$$
for $i = 1,2$. A calculation shows that  $u_1,v_1$  and $u_2,v_2$  are two fundamental systems of 
solutions to the equation  (\ref{eq.w}),
where $\rho(s) = \big\{ \varphi_1 , s\big\} = \big\{ \varphi _2, s\big\}$.
Since the solutions to  that equation  form a two-dimensional vector space, we have
$$
 u_2 = Au_1 + Bv_1, \quad v_2 = Cu_1+Dv_1  \quad (AD-BC \neq 0),
$$
and therefore
$$
 \varphi_2 = \frac{u_2}{v_2} = \frac{Au_1 + Bv_1}{ Cu_1+Dv_1} = \frac{A\varphi_1 +B}{C\varphi_1 +D}.
$$
\end{proof}

\begin{example} \label{ex.sch1}
Using this Lemma, one immediately gets that the general solution to the equation
$$
     \big\{ \varphi (s), s\big\} = 2\lambda,
$$
where $\lambda \in \r$ is a constant, is given respectively by
$$
 \frac {{{A\rm e}^{\sqrt {-\lambda}\cdot s}} +{{B\rm e}^{-\sqrt {-\lambda}\cdot s}}}{C{{\rm e}^{\sqrt {-\lambda}\cdot s}}+D{{\rm e}^{-\sqrt {-\lambda}\cdot s}}},
  \quad
 \frac{A s+B}{C s +D}, \quad \mathrm{and} \quad
 {\frac {A\sin ( \sqrt {\lambda}\cdot s ) +B\cos ( \sqrt {
\lambda}s ) }{C\sin ( \sqrt {\lambda}\cdot s ) + D \cos ( \sqrt {\lambda}\cdot s ) }},
$$
depending on  $\lambda <0$,  $\lambda = 0$ or $\lambda > 0$.
\end{example}

%______________________________________________________ 

\newpage

%\printindex


\begin{thebibliography}{10}
%

\bibitem{AcampoPapadopoulos}  N. A'Campo and A. Papadopoulos.  On Klein's  so-called Non-Euclidean Geometry. 
In : \emph{Lie and Klein: The Erlangen program and its impact in mathematics and in physics}. (L. Ji and A. Papadopoulos, ed.), 
European Mathematical Society Publishing house, Zrich, 2014 (to appear).
 

\bibitem{Akbar} H. Akbar-Zadeh, {Sur les espaces de Finsler \`a courbures sectionnelles constantes}, 
Acad. Roy. Belg. Bull. Cl. Sci. (5) 74 (10) (1988),  281Ð322. 

\bibitem{Alvarez} çlvarez Paiva 
{Symplectic geometry and Hilbert's fourth problem.} J. Differential Geom. 69 (2005), no. 2, 353Ð378. 

\bibitem{Auslander}  L. Auslander {On curvature in Finsler geometry}. Trans. Amer. Math. Soc. 79 (1955), 378Ð388.  

 
\bibitem{BCS} D. Bao, S.-S. Chern, Z. Shen,  {\it An introduction to Riemann-Finsler geometry}. Graduate Texts in Mathematics, 200. Springer-Verlag, New York, 2000.

\bibitem{Berwald1926} L. Berwald,  
{Untersuchung der Krmmung allgemeiner metrischer Rume auf Grund des in ihnen herrschenden Parallelismus.} 
Math. Z. 25 (1926), no. 1, 40Ð73. 

\bibitem{Berwald1929} L. Berwald,  
{ber die n-dimensionalen Geometrien konstanter Krmmung, in denen die Geraden die krzesten sind.} Math. Z. 30 (1929), no. 1, 449Ð469. 

\bibitem{Berwald1929b} L. Berwald,  
{\"Uber eine charakteristische Eigenschaft der allgemeinen Rume konstanter Kr\"ummung mit geradlinigen Extremalen.}
Monatshefte f. Math. 36, 315-330 (1929). 



\bibitem{Benzecri}
J.-P. Benz\'ecri, 
{Sur les vari\'et\'es localement affines et projectives.} {Bull. Soc. Math. France} 88 (1960), 229--332.


\bibitem{Bryant1997}  R. Bryant,  
{Projectively flat Finsler 2-spheres of constant curvature.}
Selecta Math. (N.S.) 3 (1997), no. 2, 161Ð203.

\bibitem{Busemann1950} H. Busemann,   {The geometry of Finsler spaces.} 
Bull. Amer. Math. Soc. 56, (1950). 5Ð16

\bibitem{Busemann1955} H. Busemann,   {\it The geometry of geodesics}, Academic Press  (1955), reprinted by Dover in 2005.

\bibitem{BusemannKelly}
H. Busemann and P. J. Kelly, {\it Projective Geometry and Projective metrics}. Academic Press, New York, 1953.

\bibitem{ChernShen}  S.-S. Chern, Z. Shen,  {\it Riemann-Finsler geometry}. Nankai Tracts in Mathematics, 6. World Scientific Publishing, 2005.
 
\bibitem{Crampin2011} M.  Crampin, 
{Some remarks on the Finslerian version of Hilbert's fourth problem.} Houston J. Math. 37 (2011), no. 2, 369Ð391.

\bibitem{Crampin2013} M.  Crampin, T. Mestdag and D.J. Saunders
{Hilbert forms for a Finsler metrizable projective class of sprays.}
Differential Geom. Appl. 31 (2013), no. 1, 63Ð79. 


\bibitem{Crampon2013}  M. Crampon
{The geodesic flow of Finsler and Hilbert
geometries}, Handbook of Hilbert geometry, 
this volume, European Mathematical Society, Z\"urich, 2014.
% (ed. G. Besson, A. Papadopoulos and M. Troyanov)

\bibitem{Dazord} P. Dazord,  
{Varits finslriennes de dimension paire $\delta$-pinces.}
C. R. Acad. Sci. Paris Sr. A-B 266 1968 

\bibitem{EgloffPhD}
D. Egloff,  {\it Some new developments in Finsler geometry}. 
Ph.D.  Thesis, University of  Freiburg (Switzerland), 1995.

\bibitem{Egloff1997}
D. Egloff,  {Uniform Finsler Hadamard manifolds}. 
Annales de l'institut Henri Poincar (A) Physique thorique, 66 no. 3 (1997), p. 323-357. 

\bibitem{Finsler} P. Finsler, 
{ber Kurven und Flchen in allgemeinen Rumen}.   
[J D] Gttingen, Zrich: O. Fssli, 120 S. $8^{\circ}$ (1918).

\bibitem{Funk1929} P. Funk 
{ber Geometrien, bei denen die Geraden die Krzesten sind}, 
Math. Annalen 101, 226Ð237 (1929)

\bibitem{Funk1936} P. Funk 
{ber zweidimensionale Finslersche Rume, insbesondere ber solche mit 
geradlinigen Extremalen und positiver konstanter Krmmung.}
Math. Z. 40 (1936), no. 1, 86Ð93.

\bibitem{Funk1970} P. Funk 
{\it Variationsrechnung und ihre Anwendung in Physik und Technik}. 
Die Grundlehren der mathematischen Wissenschaften, 
Band 94 Springer-Verlag, Berlin-New York 1970.


\bibitem{Guo2012} E. Guo, X. Mo and X. Wang 
{On projectively flat Finsler metrics.}
in Differential geometry: Under the influence of S.-S. Chern. Somerville, MA: International Press; Beijing: 
Higher Education Press. Advanced Lectures in Mathematics (ALM) 22,  (2012) 77-88. 

\bibitem{RGuo2013} R.  Guo, 
{Characterizations of hyperbolic geometry among Hilbert geometries}, Handbook of Hilbert geometry, 
this volume, European Mathematical Society, Z\"urich, 2014.


\bibitem{Hilbert} D. Hilbert,  {\"Uber die gerade Linie als k\"urzestes Verbindung zweier Punkte},
Math. Ann. XLVI. 91-96 (1895).

\bibitem{Hamel} G. Hamel {ber die Geometrieen, in denen die Geraden die Krzesten sind.} 
Math. Ann. 57 (1903), no. 2, 231Ð264. 

\bibitem{Kim2009} C.-W- Kim and K. Min 
{Finsler metrics with positive constant flag curvature.}
Arch. Math. (Basel) 92 (2009), no. 1, 70Ð79.

\bibitem{Landsberg1908} G. Landsberg, 
{\"Uber die Kr\"ummung in der Variationsrechnung.}
Math. Ann. 65, 313-349 (1908).

\bibitem{Marquis2013} L. Marquis,
{Around groups in Hilbert Geometry}, Handbook of Hilbert geometry, 
this volume, European Mathematical Society, Z\"urich, 2014.



\bibitem{Matsumoto80} M. Matsumoto, 
{Projective changes of Finsler metrics and projectively ßat Finsler spaces}. Tensor, N. S., 34 (1980), 303Ð315.


 

\bibitem{Matveev2006} V. Matveev
{Geometric explanation of the Beltrami theorem.}  
Int. J. Geom. Methods Mod. Phys. 3 (2006), no. 3, 623Ð629. 


\bibitem{MatveevTroyanov} V. Matveev and  M. Troyanov
{The Binet-Legendre Metric in Finsler Geometry}.
Geometry \& Topology 16 (2012) 2135Ð2170.


\bibitem{Mo2011}  X. Mo {On Some Finsler Metrics of Non-positive Constant Flag Curvature}.
Results in Mathematics 60  (2011)  pp 475-485.

 
\bibitem{Moalla} F. Moalla, 
{Sur quelques thormes globaux en gomtrie finslrienne. } 
Ann. Mat. Pura Appl. (4) 73 1966 319Ð365.  


\bibitem{Okada} T. Okada 
{On models of projectively flat Finsler spaces of constant negative curvature}, Tensor, N.S. 40, 117Ð123 (1983). 

 
\bibitem{PT1} A. Papadopoulos and M. Troyanov, {Weak metrics on Euclidean domains}, 
JP Journal of Geometry and Topology,  Volume 7, Issue 1 (March 2007), pp. 23-44.

\bibitem{WeakFinsler} A. Papadopoulos and M. Troyanov
{Weak Finsler Strutures and the Funk Metric.}
Math. Proc. Camb. Phil. Soc. (2009), \textbf{147}, 419-437.
%doi:10.1017/S0305004109002461

\bibitem{Hsym} A. Papadopoulos and M. Troyanov
{Harmonic symmetrization of convex sets and of Finsler structures, with applications to Hilbert geometry}.  
Expo. Math. 27 (2009), no. 2, 109Ð124. 


\bibitem{Papadopoulos-Hilbert} A. Papadopoulos, On Hilbert's Fourth Problem, Handbook of Hilbert geometry, 
this volume, European Mathematical Society, Z\"urich, 2014.


\bibitem{PT2013a} A. Papadopoulos and M. Troyanov
{Weak Minkowski Spaces}, Handbook of Hilbert geometry, 
this volume, European Mathematical Society, Z\"urich, 2014.  


\bibitem{PT2013b} A. Papadopoulos and M. Troyanov
{A survey of Funk geometry}, Handbook of Hilbert geometry, 
this volume, European Mathematical Society, Z\"urich, 2014.



\bibitem{Petersen}  P. Petersen 
{\it  Riemannian geometry}. 
Graduate Texts in Mathematics, 171. Springer, New York, (1998--2006).


\bibitem{Rademacher2004} H.-B. Rademacher 
{A sphere theorem for non-reversible Finsler metrics. }
Math. Ann. 328 (2004), no. 3, 373Ð387. 

\bibitem{Rapcsak} A. Rapcsk,  
{ber die bahntreuen Abbildungen metrischer Rume. }    
Publ. Math. Debrecen 8 1961 285Ð290. 

\bibitem{Rund59} H. Rund, 
    {\it The differential geometry of Finsler spaces}. 
     Die Grundlehren der Mathematischen Wissenschaften, Bd. 101 
     Springer-Verlag, Berlin-Gttingen-Heidelberg 1959
     
     
\bibitem{Sharpe}   R. W. Sharpe
{\it Differential geometry, 
Cartan's generalization of Klein's Erlangen program.}
Graduate Texts in Mathematics, 166. Springer-Verlag, New York, 1997.     

     
\bibitem{Shen2001a} Z. Shen, 
{\it Differential geometry of spray and Finsler space} 
Kluwer Academic Publishers (2001). 

\bibitem{Shen2001b} Z. Shen, 
{\it Lectures on Finsler geometry}. World Scientific Publishing Co., Singapore, 2001. 
 
 \bibitem{Shen2001d} Z. Shen, 
{Funk Metric and R-Flat Sprays} 
arXiv:math/0109037v1

\bibitem{Shen2001c} Z. Shen, {Geometric meanings of curvatures in Finsler geometry.}
in Proceedings of the 20th winter school ``Geometry and physics", (Srn 2000),
Rend. Circ. Mat. Palermo (2) Suppl. No. 66 (2001), 165Ð178. 

\bibitem{Shen2002a}  Z. Shen, 
{Two-dimensional Finsler metrics with constant flag curvature.}
Manuscripta Math. 109 (2002), no. 3, 349Ð366.v

\bibitem{Shen2003a} Z. Shen, 
{Projectively flat Finsler metrics of constant flag curvature.} 
Trans. Amer. Math. Soc. 355 (2003), no. 4, 1713Ð1728. 

\bibitem{Shen2003b} Z. Shen, 
{Projectively flat Randers metrics of constant flag curvature}, Math. Ann. 325 (2003), 19--30.

\bibitem{Sz1} Z. I. Szab\'o:  {Positive definite Berwald spaces (Structure theorems)}.
Tensor N. S. {\bf 35}(1981) 2539.

%\bibitem{Socie1} E. Soci\'e-M\'ethou,  {Comportements asymptotiques et rigidit\'es des g\'eom\'etries de Hilbert}, 
%PhD thesis, University of Strasbourg, 2000. 

%\bibitem{Socie2}  E. Soci\'e-M\'ethou, {Behaviour of distance functions in HilbertÐFinsler geometry}, 
%differential Geometry and its Applications \textbf{20}, Issue 1 (2004) 1--10.

 \bibitem{Thompson} A. C. Thompson,  {\it Minkowski geometry}. 
 Encyclopedia of Mathematics and its Applications, 63. Cambridge University Press, Cambridge, 1996.
 
 \bibitem{Underhill}  A. Underhill
  {Invariants of the Function $F(x, y, x', y')$ in the Calculus of Variations}.
 Transactions of the American Mathematical Society , Vol. 9, No. 3 (1908), pp. 316-338.

 \bibitem{Varga}  O. Varga, 
{Zur Begrndung der Minkowskischen Geometrie. } 
Acta Univ. Szeged. Sect. Sci. Math. 10, (1943). 149Ð163.

 \bibitem{Varga1949}  O. Varga, {\"Uber das Kr\"ummungsmass in Finslerschen R\"aumen}.  
Publ. Math., Debrecen 1, 116-122 (1949).

\bibitem{XU} B. Xu and B. Li,
{On a class of projectively flat Finsler metrics with flag curvature $K=1$. }
Differential Geom. Appl. 31 (2013), no. 4, 524Ð532. 

\bibitem{Yamada} S. Yamada,
{Convex bodies in Euclidean and Weil-Petersson geometries}, to appear
in Proc. Amer. Math. Soc. (arXiv:1110.5022v2).

 
\bibitem{Zaustinsky} E. M. Zaustinsky, 
{\it Spaces with nonsymmetric distance}, Mem. Amer. Math. 
Soc. \textbf{34}, 1959.

\bibitem{zhou2013}  L. Zhou, 
Projective spherically symmetric Finsler metrics with constant flag curvature in $\mathbb{R}^n$.  
Geom. Dedicata 158 (2012), 353Ð364.

\end{thebibliography}
\end{document}